\documentclass[11pt,a4paper]{amsart}[2000]

\usepackage{amsmath}
\usepackage{amssymb}
\usepackage{amsthm}
\usepackage[hang]{footmisc}
\usepackage[all]{xypic}
\usepackage{hyperref}
\usepackage{color}
\usepackage{mathtools}
\usepackage{enumitem}

\title[Purely infinite locally compact Hausdorff \'{e}tale groupoids]{Purely infinite locally compact Hausdorff \'{e}tale groupoids and their $C^*$-algebras}

\thanks{}

\allowdisplaybreaks

\theoremstyle{plain}

\newtheorem{Thm}{Theorem}[section]

\theoremstyle{definition}

\newtheorem{Exl}[Thm]{Example}
\newtheorem{Rmk}[Thm]{Remark}

\theoremstyle{plain}
\newtheorem{thm}[Thm]{Theorem}
\newtheorem{lem}[Thm]{Lemma}
\newtheorem{cor}[Thm]{Corollary}
\newtheorem{prop}[Thm]{Proposition}

\theoremstyle{definition}
\newtheorem{defn}[Thm]{Definition}
\newtheorem{ques}[Thm]{Question}

\newtheorem{innercustomthm}{Theorem}
\newenvironment{customthm}[1]
{\renewcommand\theinnercustomthm{#1}\innercustomthm}
{\endinnercustomthm}

\newcommand{\R}{{\mathbb R}}
\newcommand{\N}{{\mathbb N}}
\newcommand{\Z}{{\mathbb Z}}

\newcommand{\F}{{\mathbb F}}

\newcommand{\supp}{\mathrm{supp}}

\numberwithin{equation}{section}

\newcommand{\id}{\mathrm{id}}

\newcommand{\CA}[0]{\mathcal{A}} 
\newcommand{\CC}[0]{\mathcal{C}} \newcommand{\CD}[0]{\mathcal{D}}
 \newcommand{\CF}[0]{\mathcal{F}}
\newcommand{\CG}[0]{\mathcal{G}} 
\newcommand{\CI}[0]{\mathcal{I}} 
\newcommand{\CK}[0]{\mathcal{K}} 
 
\newcommand{\CO}[0]{\mathcal{O}}

 \newcommand{\CV}[0]{\mathcal{V}}
\newcommand{\CW}[0]{\mathcal{W}}

\newcommand{\diag}[0]{\operatorname{diag}}

\newcommand{\GU}[0]{\CG^{(0)}}

\newcommand*\diff{\mathop{}\!\mathrm{d}}

\newtheorem{lemma}[Thm]{Lemma}

\theoremstyle{definition}

\numberwithin{equation}{Thm}

\begin{document}
	
\author{Xin Ma}
\email{xma29@buffalo.edu}
\address{Department of Mathematics,
	State University of New York at Buffalo,
	Buffalo, NY, 14260}
%\curraddr{Department of Mathematics,
%	Sate University of New York at Buffalo,
%	Buffalo, NY, 14260}

%    General info
\subjclass[2010]{37B05, 46L35}
\keywords{Pure infiniteness of groupoids, Groupoid semigroup, Reduced groupoid $C^*$-algebras}

\date{Sep 23, 2020.}

%\author{Jianchao Wu}
	
\begin{abstract}
In this paper, we introduce properties including groupoid comparison, pure infiniteness and paradoxical comparison as well as a new algebraic tool called groupoid semigroup for locally compact Hausdorff \'{e}tale groupoids. We show these new tools help establishing pure infiniteness of reduced groupoid $C^*$-algebras. As an application, we show a dichotomy of stably finiteness against pure infiniteness for reduced groupoid $C^*$-algebras arising from locally compact Hausdorff \'{e}tale minimal topological principal groupoids. This generalizes the  dichotomy obtained by  B\"{o}nicke-Li  and Rainone-Sims. We also study the relation among our paradoxical comparison, $n$-filling property and locally contracting property appeared in the literature for locally compact Hausdorff \'{e}tale groupoids. 
\end{abstract}
\maketitle

\section{Introduction}
Nowadays, deep connections are known between $C^*$-algebras and topological groupoids. In particular,  locally compact \'{e}tale groupoids have long been an important source of examples and motivation for the study of $C^\ast$-algebras via the construction of the reduced groupoid $C^\ast$-algebra $C_r^\ast(\CG)$  from a locally compact \'{e}tale groupoid $\CG$. Purely infinite simple  $C^*$-algebras were introduced by Cuntz in \cite{Cuntz} in which a simple $C^*$-algebra $A$ is called \textit{purely infinite} if every non-zero hereditary sub-$C^*$-algebra of $A$ contains an infinite projection. The class of simple separable nuclear purely infinite $C^*$-algebras, called \textit{Kirchberg algebras}, is of particular interest because of the classification of these algebras by $K$- or $KK$-theory obtained by Kirchberg and Phillips (see \cite{Phillips} for example). Many natural examples of Kirchberg algebras, for example, the Cuntz algebras $\CO_n$ for $2\leq n\leq \infty$, can be obtained from groupoids (see \cite{Renault}). In this paper, we restrict our discussion to locally compact Hausdorff \'{e}tale groupoids. We say a $C^*$-algebra $A$ has a \textit{groupoid model} if there is a locally compact Hausdorff \'{e}tale groupoid $\CG$ such that $A\simeq C^*_r(\CG)$. If $\CG$ can be chosen to be a transformation groupoid, we say $A$ has a \textit{dynamical model}.  Spielberg, in \cite{Sp},  showed that any Kirchberg algebra $A$ in the UCT class has a groupoid model in the sense that there is a directed graph $E$, as a mixture of directed $1$-graph and $2$-graphs, generating a locally compact Hausdorff \'{e}tale second countable groupoid $\CG_E$ such that $A\simeq C^*_r(\CG_E)$. 

In \cite{K-Rord}, Kirchberg and R{\o}rdam generalized Cuntz's original notion of pure infiniteness to not necessarily simple $C^*$-algebras by using the Cuntz subequivalence relation. In addition, in \cite{Kir-Rord}, they introduced strongly pure infiniteness for a general $C^*$-algebra, which was shown there to be equivalent to $\CO_\infty$-absorption, i.e., $A\otimes \CO_\infty\simeq A$ if the $C^*$-algebra $A$ is nuclear and separable. It then has been  verified by Kirchberg and R{\o}rdam in \cite{Kir-Rord} that strongly pure infiniteness and pure infiniteness are equivalent if the $C^*$-algebra is simple or of real rank zero. Furthermore, in \cite{P-R} Pasnicu and R{\o}rdam established this equivalence for $C^*$-algebras with the ideal property (IP).

 However, so far it is not known whether all strongly purely infinite or purely infinite $C^*$-algebras in the sense of Kirchberg and R{\o}rdam have a locally compact Hausdorff \'{e}tale groupoid model except Spielberg's result for Kirchberg algebras mentioned above. To study this question, it is believed that one first needs to come up with a notion describing pure infiniteness or paradoxicality within the category of locally compact Hausdorff \'{e}tale groupoids as a regularity property implying the pure infiniteness of the reduced groupoid $C^*$-algebras. To the best knowledge of the author, the earliest systematic study of this question was initiated by Anantharaman-Delaroche in \cite{A-D} through the introduction of a property named \textit{locally contracting}. See also \cite{L-S} for the dynamical version of locally contracting called \textit{local boundary action}. It was proved in \cite{A-D} that if a locally compact Hausdorff \'{e}tale groupoid $\CG$ is locally contracting then every non-zero hereditary sub-$C^*$-algebra of $C^*_r(\CG)$ contains an infinite projection. Therefore, if $C^*_r(\CG)$ is also simple, which holds when $\CG$ is minimal and topologically principal, proved in \cite{Renault} as well as \cite{B-L}, then $C^*_r(\CG)$ is purely infinite. Motivated by $n$-filling actions defined in \cite{J-R}, Suzuki in \cite{YS} introduced the  $n$-\textit{filling} property for locally compact Hausdorff \'{e}tale groupoids $\CG$ on compact spaces and show that the $n$-filling property implies that  $C^*_r(\CG)$ is purely infinite and simple under an assumption that  $\CG$ is topological principal.  Rainone-Sims in \cite{Ra-Sims} recovered this result in the case of ample groupoids. On the other hand, in the study of ample groupoids, Rainone-Sims in \cite{Ra-Sims} and B\"{o}nicke-Li in \cite{B-L}  independently generalized the paradoxicality of compact open sets for dynamical systems defined by R{\o}rdam and Sierakowski in \cite{R-S} to the setting of locally compact Hausdorff \'{e}tale ample groupoids. It was also shown in \cite{B-L} by using an algebraic argument that if a locally compact Hausdorff \'{e}tale ample groupoid $\CG$ is essentially principal and has this nice paradoxical decomposition of compact open sets in the unit space $\GU$ then $C^*_r(\CG)$ is purely infinite and thus strongly purely infinite because $C^*_r(\CG)$ in this case has the ideal property (IP). 
 
 We remark that all purely infinite $C^*$-algebras arising from the groupoids in \cite{A-D}, \cite{Kir-S}, \cite{B-L} and \cite{R-S} above have infinite projections. However, there are many purely infinite non-simple $C^*$-algebras having no non-zero projections, for example, $\CO_2\otimes C_0(\R)$. A locally compact Hausdorff \'{e}tale groupoid model of such a $C^*$-algebra, if it exists,  is necessarily not ample and has no locally contracting property. Therefore, in order to find groupoid models for a larger class of purely infinite $C^*$-algebras including such projectionless $C^*$-algebras, one needs to come up with a new framework. Motivated by this, in this paper,  we introduce a new dynamical approach based on a property called \textit{dynamical comparison}. The concept of dynamical comparison was introduced by Wilhelm Winter in 2012 and refined by David Kerr in \cite{D} to study actions of amenable groups on compact metrizable spaces. Then in \cite{M2}, the author studied dynamical comparison and introduced \textit{paradoxical comparison} for actions of non-amenable groups on compact Hausdorff spaces to obtain several dynamical criteria establishing the pure infiniteness for unital reduced crossed product $C^*$-algebras. In \cite{M3}, the author introduced a new semigroup, called \textit{the generalized type semigroup}, as a generalization of the type semigroup dating back to Tarski, to study dynamical comparison. It was shown in \cite{M3} that the dynamical comparison  relates to the almost unperforation of the generalized type semigroup. Motivated by the success of the comparison defined on compact spaces, we generalize and apply ideas of comparison to locally compact settings by studying locally compact Hausdorff \'{e}tale groupoids. In this paper, we will introduce a groupoid version of dynamical comparison, called \textit{groupoid comparison}, two types of pure infiniteness, and paradoxical comparison (see Definition 3.4, 3.5, 3.6 and 3.7). As one can see in Section 3, many examples of groupoids themselves in the literature  are purely infinite in our sense.  Therefore our approach would  serve as a unified and new framework for many attempts in the literature on the study of pure infiniteness of reduced groupoid $C^*$-algebras. On the other hand,  in \cite{Matui2} Matui  also introduced the notion of pure infiniteness for locally compact Hausdorff \'{e}tale ample groupoids on compact spaces to study the topological full group of the groupoid generated by one-sided shifts of finite type. We show in section 5 that our pure infiniteness is a higher dimensional generalization of his pure infiniteness. 
 
 We say $C_0(\GU)$ separates ideals of $C_r^*(\CG)$ if the (surjective) map $I \mapsto I\cap C_0(\GU)$ from ideals in $C^*_r(\CG)$ to ideals in $C_0(\GU)$ generated by $\CG$-invariant closed sets is injective. It was shown in \cite{B-L} that if $\CG$ is an amenable locally compact Hausdorff \'{e}tale groupoid then $C_0(\GU)$ separates ideals of $C_r^*(\CG)$ if and only if $\CG$ is essentially principal. This result is a groupoid generalization on the ideal structure for reduced crossed product $C^*$-algebras of dynamical systems obtained by Sierakowski in \cite{S}. The following is our first main result. 
 
 \begin{customthm}{A}(Theorem 6.8)
 		Let $\CG$ be a locally compact Hausdorff \'{e}tale groupoid such that $C_0(\GU)$ separates ideals of $C_r^*(\CG)$.   If $\CG$ is purely infinite  and all open sets in $\GU$ are groupoid small in the sense of Definition 6.6 then $C^*_r(\CG)$ is purely infinite. 
 \end{customthm}

As main applications, we have the following corollaries.

\begin{cor}
	Let $\CG$ be an amenable locally compact Hausdorff \'{e}tale essentially principal second countable  groupoid with finitely many $\CG$-invariant closed sets in $\GU$. If $\CG$ is purely infinite then $C^*_r(\CG)$ is nuclear separable and strongly purely infinite.
\end{cor}

On the other hand when $\CG$ is minimal then the condition that $C_0(\GU)$ separates ideals of $C_r^*(\CG)$ in the Theorem holds trivially. In addition, there is no non-trivial $\CG$-invariant closed set in $\GU$. It was also proved in \cite{Renault} as well as \cite{B-L} that if $\CG$ is minimal and topologically principal then $C^*_r(\CG)$ is simple. Therefore, we also have the following corollary for the case that $\CG$ is minimal without the assumption that $\CG$ is amenable.

\begin{cor}
		Let $\CG$ be a locally compact Hausdorff \'{e}tale minimal topological principal groupoid. If $\CG$ is purely infinite  then $C^*_r(\CG)$ is simple and (strongly) purely infinite. 
\end{cor}

It will be shown in Section 5 that our pure infiniteness is equivalent to groupoid comparison if $\CG$ is minimal and there is no $\CG$-invariant probability Borel regular measure on $\GU$. Then, as an application of Corollary 1.2, we have the following dichotomy as our second main result.

 \begin{customthm}{B}(Theorem 6.11)
 Let $\CG$ be a locally compact Hausdorff \'{e}tale minimal topological principal groupoid. Suppose $\CG$ has groupoid comparison.  Then $C^*_r(\CG)$ is either stably finite or strongly purely infinite.
 \end{customthm}

We introduce in Section 4 a new algebraic tool called \textit{groupoid semigroup} for a groupoid $\CG$, denoted by $\CW(\CG)$, which has its own interest to be studied because various comparison properties in groupoids can be coded within this semigroup, which is similar to how Cuntz semigroup codes the strict comparison in the $C^*$-setting. From this philosophy we establish in Section 5 the relation among our various comparison properties and pure infiniteness. In addition, we also use the groupoid semigroup to study the type semigroup of the groupoid $\CG$, denoted by $\CV(\CG)$, when $\CG$ is ample. While our groupoid semigroup can be regarded as an analogue of Cuntz semigroup,  the type semigroup plays a role like Murray-von Neumann semigroup.  From this direction, we have the following result.

\begin{customthm}{C}(Theorem 5.15)
		Let $\CG$ be a locally compact Hausdorff \'{e}tale ample groupoid.  Consider the following conditions.
	\begin{enumerate}[label=(\roman*)]
		\item $\CG$ is purely infinite.
		
		\item $\CG$ has paradoxical comparison.
		
		\item Every clopen set  in $\GU$ is $(2, 1)$-paradoxical in the sense of \cite{B-L}.
		
		\item $\CW(\CG)$ is purely infinite.
		
		\item $\CV(\CG)$ is purely infinite.
		
		\item $\CW(\CG)$ is almost unperforated and there is no non-trivial state on $\CW(\CG)$.
		
		\item $\CV(\CG)$ is almost unperforated and there is no non-trivial state on $\CV(\CG)$.
		
		\item $\CG$ is weakly purely infinite.
	\end{enumerate}
	Then (i)-(vii) are equivalent in general. If $\CG$ is furthermore assumed to have no global fixed unit in $\GU$, then all conditions above are equivalent to (viii).
\end{customthm}

In light of Theorem C, our Theorem B is a generalization of the dichotomy on stably finiteness against pure infiniteness obtained in \cite{B-L} and \cite{Ra-Sims}. 

 Finally, it is also interesting to investigate the range of our  pure infiniteness as a groupoid regularity property. We address this question by comparing our pure infiniteness (or paradoxical comparison) with $n$-filling and locally contracting property introduced in the literature for locally compact Hausdorff \'{e}tale groupoids. We are particularly interested in comparing our pure infiniteness with locally contracting property because it is one of the most prevalent notion establishing pure infiniteness of reduced groupoid $C^*$-algebras. First, as we will show in Section 3, a lot of examples of groupoids in the literature satisfying locally contracting are purely infinite. But it is not clear in general whether locally contracting implying pure infiniteness. On the other hand, we will show in Section 5 that  our pure infiniteness implies locally contracting in the case of locally compact Hausdorff \'{e}tale minimal ample groupoids on  compact  spaces. However, we discover in this paper that pure infiniteness does not imply locally contracting in general by establishing the following as our final main result. This result also shows that our pure infiniteness covers examples beyond the realm of locally contracting property.

\begin{customthm}{D}(Theorem 6.15)
	There exists a non-simple strongly purely infinite $C^*$-algebra $A$, for example, $\CO_2\otimes C_0(\R)$, which has no locally compact Hausdorff \'{e}tale locally contracting groupoid model but there is a  locally compact Hausdorff \'{e}tale purely infinite groupoid  $\CG$  such that $A\simeq C^*_r(\CG)$.
	\end{customthm}

For groupoids on compact spaces, we will show in Section 5 that $n$-filling is equivalent to our pure infiniteness in the minimal ample case. Therefore, our pure infiniteness is also a generalization of the $n$-filling property. Using our pure infiniteness as a bridge, we show in Section 5 that the $n$-filling property implies locally contracting in the case that groupoid is minimal ample. This answers a question in \cite{Ra-Sims}. See the sentence before Remark 9.5 in \cite{Ra-Sims}. 

Our paper is organized in the following way. In Section 2, we review some necessary concepts, definitions and preliminary results. In Section 3, we introduce groupoid comparison, paradoxical comparison and pure infiniteness of groupoids. We also list many examples of purely infinite groupoids there. In Section 4, we introduce the groupoid semigroup and list many fundamental properties of it. In Section 5, we study the relation among various notions with paradoxical flavor appeared in this paper or other literature by using the groupoid semigroup and other tools. In addition, we will study the type semigroup for ample groupoids. In Section 6, we will study reduced groupoid $C^*$-algebras by using pure infiniteness of the groupoids.

\section{Preliminaries}
In this section we recall some basic backgrounds on \'{e}tale groupoids and $C^*$-algebras. We refer to \cite{Renault} and \cite{Sims} as standard references for groupoids. 

\begin{defn}
	A \textit{groupoid} $\CG$ is a set equipped with a distinguished subset $\CG^{(2)}\subset \CG\times \CG$, called the set of \textit{composable pairs}, a product map $\CG^{(2)}\rightarrow \CG$, denoted by $(\gamma , \eta)\mapsto \gamma\eta$ and an inverse map $\CG \rightarrow \CG$, denoted by $\gamma\mapsto \gamma^{-1}$ such that the following hold
\begin{enumerate}[label=(\roman*)]
\item If $(\alpha, \beta)\in \CG^{(2)}$ and $(\beta, \gamma)\in \CG^{(2)}$ then so are $(\alpha\beta, \gamma)$ and $(\alpha, \beta\gamma)$. In addition, $(\alpha\beta)\gamma=\alpha(\beta\gamma)$ holds in $\CG$.

\item For all $\alpha \in \CG$ one has $(\gamma, \gamma^{-1})\in \CG^{(2)}$ and $(\gamma^{-1})^{-1}=\gamma$. 

\item For any $(\alpha, \beta)\in \CG^{(2)}$ one has $\alpha^{-1}(\alpha\beta)=\beta$ and $(\alpha\beta)\beta^{-1}=\alpha$.
\end{enumerate}	
Every groupoid is equipped with a subset $\GU=\{\gamma\gamma^{-1}: \gamma\in \CG\}$ of $\CG$. We refer to elements of $\GU$ as \textit{units} and to $\GU$ itself as the \textit{unit space}. We define two maps $s, r: \CG\rightarrow \GU$ by $s(\gamma)=\gamma^{-1}\gamma$ and $r(\gamma)=\gamma\gamma^{-1}$, respectively, in which $s$ is called the \textit{source} map and $r$ is called the \textit{range} map.
\end{defn}

When a groupoid $\CG$ is endowed with a locally compact Hausdorff topology under which the product and inverse maps are continuous, the groupoid $\CG$ is called a locally compact Hausdorff groupoid. A locally compact Hausdorff groupoid $\CG$ is called \textit{\'{e}tale} if the range map $r$ is a local homeomorphism, which means for any $\gamma\in \CG$ there is an open neighborhood $U$ of $\gamma$ such that $r(U)$ is open and $r|_U$ is a homeomorphism. It can be verified that if $r$ is a local homeomorphism then so is the source map $s$. An open set $U$ in $\CG$ is called an open \textit{bisection} if the restriction of the source map $s|_U: U\rightarrow s(U)$ and the range map $r|_U: U\rightarrow r(U)$ on $U$ are both homeomorphisms onto open subsets of $\GU$. It is not hard to see a locally compact Hausdorff groupoid is \'{e}tale if and only if its topology has a basis consisting of open bisections. We say a  locally compact Hausdorff \'{e}tale groupoid $\CG$ is \textit{ample} if its topology has a basis consisting compact open bisections.  

\begin{Exl}
	Let $X$ be a locally compact Hausdorff space and $\Gamma$ be a discrete group. Then any continues action $ \Gamma\curvearrowright X$ induces a locally compact Hausdorff \'{e}tale groupoid 
	\[\CG_{\Gamma\curvearrowright X}\coloneqq\{(\gamma x, \gamma, x), \gamma\in\Gamma, x\in X\}\]
equipped with the relative topology as a subset of $X\times \Gamma\times X$. In addition, $(\gamma x, \gamma ,x)$ and $(\beta y, \beta, y)$ are composable only if $\beta y=x$ and  
\[(\gamma x, \gamma ,x)(\beta y, \beta, y)=(\gamma\beta y, \gamma\beta y, y ).\]
One also defines $(\gamma x, \gamma, x)^{-1}=(x, \gamma^{-1}, \gamma x)$ and announces that $\GU\coloneqq \{(x, e_\Gamma, x): x\in X\}$. It is not hard to verify that $s(\gamma x, \gamma, x)=x$ and $r(\gamma x, \gamma, x)=\gamma x$. The groupoid $\CG_{\Gamma\curvearrowright X}$ is called a \textit{transformation groupoid}.
\end{Exl}

The following are several basic properties of locally compact Hausdorff \'{e}tale groupoids whose proofs could be found in \cite{Sims}.

\begin{prop}
	Let $\CG$ be a locally compact Hausdorff \'{e}tale groupoid. Then $\GU$ is a clopen set in $\CG$.
\end{prop}

\begin{prop}
Let $\CG$ be a locally compact Hausdorff \'{e}tale groupoid. Suppose $U$ and $V$ are open bisections in $\CG$. Then $UV=\{\alpha\beta\in \CG: (\alpha, \beta)\in \CG^{(2)}\cap U\times V\}$ is also an open bisection.
\end{prop}

\begin{prop}
	Let $\CG$ be a locally compact Hausdorff \'{e}tale groupoid. Let $\gamma\in \CG$ be an element such that $s(\gamma)\neq r(\gamma)$. Then there is an open bisection $O$ containing $\gamma$ such that $s(O)\cap r(O)=\emptyset$.
\end{prop}

For any set $D\subset \GU$,  Denote by 
\[\CG_D\coloneqq \{\gamma\in \CG: s(\gamma)\in D\},\ \CG^D\coloneqq \{\gamma\in \CG: r(\gamma)\in D\},\ \text{and}\ \ \CG^D_D\coloneqq\CG^D\cap \CG_D.\]
For the singleton case $D=\{u\}$, we write $\CG_u$, $\CG^u$ and $\CG^u_u$ instead for simplicity. Each $\CG^u_u$ is a group, which is called the \textit{isotropy} at $u$. We say a groupoid $\CG$ is \textit{principal} if all isotropy groups are trivial, i.e., $\CG_u^u=\{u\}$ for all $u\in \GU$. We say a groupoid $\CG$ is \textit{topologically principal} if  the set $\{u\in \GU: \CG^u_u=\{u\}\}$ is dense in $\GU$. A subset $D$ in $\GU$ is called $\CG$-\textit{invariant} if $r(\CG D)=D$, which is equivalent to the condition $\CG^D=\CG_D$. Note that 
$\CG|_D\coloneqq \CG_D^D$ is a subgroupoid of $\CG$ with the unit space $D$ if $D$ is a $\CG$-invariant set in $\GU$. A groupoid $\CG$ is called \textit{minimal} if there are no proper non-trivial closed $\CG$-invariant subsets in $\GU$. We say a groupoid $\CG$ is \textit{essentially principal} if the closed subgroupoid $\CG|_D$ is topological principal for every  $\CG$-invariant closed set $D$ in $\GU$.

Let $\CG$ be a locally compact Hausdorff \'{e}tale groupoid. We define a convolution product on $C_c(\CG)$ by 
\[(f*g)(\gamma)=\sum_{\alpha\beta=\gamma}f(\alpha)g(\beta)\] and an involution by
\[f^*(\gamma)=\overline{f(\gamma^{-1})}.\]
These two operations make $C_c(\CG)$ a $*$-algebra. Then the reduced groupoid $C^*$-algebra $C^*_r(\CG)$ is defined to be the completion of $C_c(\CG)$ with respect to the norm $\|\cdot\|_r$ induced by all regular representation $\pi_u$ for $u\in \GU$, where $\pi_u: C_c(\CG)\rightarrow B(\ell^2(\CG_u))$ is defined by $\pi_u(f)\eta=f*\eta$ and $\|f\|_r=\sup_{u\in \GU}\|\pi_u(f)\|$. 
It is well known that there is a $C^*$-algebraic embedding $\iota : C_0(\GU)\rightarrow C^*_r(\CG)$. On the other hand, $E_0: C_c(\CG)\rightarrow C_c(\GU)$ defined by $E_0(a)= a|_{\GU}$ extends to a faithful canonical conditional expectation $E: C^*_r(\CG)\rightarrow C_0(\GU)$ satisfying $E(\iota(f))=f$ for any $f\in C_0(\GU)$ and $E(\iota(f)a\iota(g))=fE(a)g$ for any $a\in C^*_r(\CG)$ and $f, g\in C_0(\GU)$.

As a typical example, it can be verified that for the transformation groupoid in Example 2.2,  the reduced groupoid $C^*$-algebra is  isomorphic to the reduced crossed product $C^*$-algebra of the dynamical system. The following are  some standard facts on reduced groupoid $C^*$-algebras that could be found in \cite{Sims}. Throughout the paper, the notation $\supp(f)$ for a function $f$ on a topological space $X$ denotes the open support $\{x\in X: f(x)\neq 0\}$ of $f$. We say an open set $O$ in a topological space $X$ is \textit{precompact} if $\overline{O}$ is compact.

\begin{prop}
	Let $\CG$ be a locally compact Hausdorff \'{e}tale groupoid. Any $f\in C_c(\CG)$ can be written as a sum $f=\sum_{i=0}^nf_i$ such that  there are precompact open bisections $V_0, \dots, V_n$ such that $V_0\subset \GU$ and $V_i\cap \GU=\emptyset$ for all $0<i\leq n$ as well as  $\overline{\supp(f_i)}\subset V_i$ for any $0\leq i\leq n$.
\end{prop}

\begin{prop}
		Let $\CG$ be a locally compact Hausdorff \'{e}tale groupoid. Suppose $U, V$ are open bisections and $f, g\in C_c(\CG)$ such that $\overline{\supp(f)}\subset U$ and $\overline{\supp(g)}\subset V$. Then $\overline{\supp(f*g)}\subset U\cdot V$ and for any $\gamma=\alpha\beta\in U\cdot V$ one has $(f*g)(\gamma)=f(\alpha)g(\beta)$.
\end{prop}

Let $\CG$ be a locally compact Hausdorff \'{e}tale groupoid. Suppose $U$ is an open bisection and $f\in C_c(\CG)_+$ such that $\overline{\supp(f)}\subset U$. Define functions $s(f), r(f)\in C_0(\GU)$ by $s(f)(s(\gamma))=f(\gamma)$ and $r(f)(r(\gamma))=f(\gamma)$ for $\gamma\in \supp(f)$. Since $U$ is a bisection, so is $\supp(f)$. Then the functions $s(f)$ and $r(f)$ are well-defined functions on $s(\supp(f))$ and $r(\supp(f))$, respectively.  Note that $s(f)=(f^**f)^{1/2}$ and $r(f)=(f*f^*)^{1/2}$.

For Cuntz comparison, we refer to \cite{A-P-T} and \cite{NCP} as standard references. Let $A$ be a $C^\ast$-algebra. We write $M_\infty(A)=\bigcup_{n=1}^\infty M_n(A)$ (viewing $M_n(A)$ as an upper left-hand corner in $M_m(A)$ for $m>n$). Let $a,b$ be two positive elements in $M_n(A)_+$ and $M_m(A)_+$, respectively. Set $a\oplus b= \textrm{diag}(a,b)\in
M_{n+m}(A)_+$, and write $a\precsim_A b$ if there exists a sequence $(r_n)$ in $M_{m,n}(A)$ with $r_n^\ast
br_n\rightarrow a$.  If there is no confusion, we usually omit the subscript $A$ by writing $a\precsim b$ instead. We write $a\sim b$ if $a\precsim b$
and $b\precsim a$. Note that $a^*a\sim aa^*$ for any $a\in A$ and $a\sim a^{1/2}$ for any $a\in A_+$. These show that $s(f)\sim r(f)$ in $C^*_r(\CG)$ for any $f\in C_c(\CG)_+$ with $\overline{\supp(f)}\subset U$ for some open bisection $U$ in $\CG$.

A non-zero positive element $a$ in $A$ is said to be\textit{ properly infinite} if
$a\oplus a\precsim a$. A $C^\ast$-algebra $A$ is said to be \textit{purely infinite} if there are no characters
on $A$ and if, for every pair of positive elements $a,b\in A$ such that $b$ belongs to the closed ideal in $A$ generated by $a$, one has $b\precsim a$. It was proved in \cite{K-Rord} that a $C^\ast$-algebra $A$ is purely infinite if and
only if every non-zero positive element $a$ in $A$ is properly infinite. In addition, in \cite{Kir-Rord}, Kirchberg and R{\o}rdam also introduced a stronger version of pure infiniteness for $C^*$-algebras called \textit{strongly pure infiniteness}. See Definition 5.1 in \cite{Kir-Rord}. We remark that strongly pure infiniteness for a $C^*$-algebra $A$ is equivalent to $A\otimes \CO_\infty\simeq A$ if $A$ is separable and nuclear. It was proved in \cite{P-R} that strongly pure infiniteness is equivalent to pure infiniteness if $A$ has the \textit{ideal property} (IP), which says that projections separate ideals in $A$. This class including purely infinite $C^*$-algebras with real rank zero and thus purely infinite simple $C^*$-algebras. See also \cite{Kir-Rord}.

We refer the definition of ($2$-)\textit{quasitraces} to \cite{B-K}. A quasitrace is a map $\tau$ defined on positive elements of a $C^*$-algebra $A$ by $\tau: A_+\to [0, \infty]$ satisfying 
\begin{enumerate}[label=(\roman*)]
\item $\tau(d^*d)=\tau(dd^*)$ for any $d\in A$.
\item $\tau(a+b)=\tau(a)+\tau(b)$ for all commuting positive elements $a, b\in A_+$.
\end{enumerate}
If a quasitrace $\tau$ can be extended to a quasitrace $\tau_2$ on $M_2(A)_+$ with $\tau_2(a\otimes e_{11})=\tau(a)$ for all $a\in A_+$ then we call $\tau$ a $2$-quasitrace. A quasitrace is called \textit{trivial} if  it takes only the value $0$ and $\infty$. A quasitrace $\tau$ is called \textit{semi-finite} if $\{a\in A: \tau(a^*a)<\infty\}$ is dense in $A$ and \textit{bounded} if $\tau(A_+)\subset [0, \infty)$. Note that a bounded quasitrace is automatically semifinite. A quasitrace $\tau$ is called \textit{faithful} if $\tau(a)>0$ whenever $a\in A_+\setminus \{0\}$. A quasitrace $\tau$ is called \textit{lower semi-countinuous} if $\tau(a)=\sup_{\epsilon>0}\tau((a-\epsilon)_+)$ for $a\in A_+$. A $C^*$-algebra $A$ is called \textit{traceless} if there is no non-trivial lower semi-continuous $2$-quasitrace on $A$. It was shown in Remark 2.27(viii) in \cite{B-K} that a simple $C^*$-algebra $A$ is stably finite if and only if there exists a faithful semi-finite lower semi-continuous 2-quasitrace on $A$. This is a non-unital generalization of the celebrated result of Cuntz on the characterization of stably finiteness on unital simple $C^*$-algebras (see \cite{Cu}).

Finally, throughout the paper, we write $A\sqcup B$ to indicate that the union of sets $A$ and $B$ is a disjoint union. In addition, we denote by $\bigsqcup_{i\in I}A_i$ for the disjoint union of the family $\{A_i: i\in I\}$. In addition, all groupoids under consideration in this paper have infinite unit spaces.

\section{Pure infiniteness and paradoxical comparison}
In this section, we study various properties characterizing pure infiniteness of locally compact Hausdorff \'{e}tale groupoids. We first recall comparison in dynamical systems of discrete groups on locally compact Hausdorff spaces.

\begin{defn}(\cite[Definition 3.1]{D})
		Let $\alpha: \Gamma\curvearrowright X$ be an action of a discrete group $\Gamma$ on a locally compact Hausdorff space $X$. Let $F$ be a compact set in $X$ and $O$ a non-empty open subset of $X$. We write $F\prec O$ if there exists a finite collection $\mathcal{U}$ of open subsets of $X$ which cover $F$, an $s_U\in \Gamma$ for each $U\in \mathcal{U}$ such that the images $s_UU$ for $U\in \mathcal{U}$ are pairwise disjoint subsets of $O$. In addition, for open sets $U,V$, we write $U\prec V$ if $F\prec V$ holds whenever $F$ is a compact subset of $U$.
\end{defn}

 The following definition is a natural groupoid analogue of the subequivalence relation of open sets above. 

\begin{defn}
	Let $\CG$ be a locally compact Hausdorff \'{e}tale groupoid.
	\begin{enumerate}[label=(\roman*)]
		\item Let $K$ be a compact subset of $\CG^{(0)}$ and $V$ an open subset of $\CG^{(0)}$. We  write $K \prec_{\CG} V$ if there are open bisections $A_1, \ldots, A_n$  such that $K \subset \bigcup_{i=1}^{n} s(A_i)$, $\bigsqcup_{i=1}^{n} r(A_i) \subset V$. 
		\item Let $U,V$ be open subsets of $\CG^{(0)}$. We  write $U \prec_{\CG} V$ if $K \prec_{\CG} V$ for every compact subset $K \subset U$. 
	\end{enumerate}
	\end{defn}

\begin{Rmk}
	We remark that for transformation groupoids $\CG_{\Gamma\curvearrowright X}$, our Definition 3.2 for $\CG_{\Gamma\curvearrowright X}$ coincides with Definition 3.1 for its generating dynamical system $\Gamma\curvearrowright X$. 
\end{Rmk}

We say a Borel regular measure $\mu$ on $\GU$ is $\CG$-invariant if $\mu(r(U))=\mu(s(U))$ for any open bisection $U$ in $\CG$. Denote by $M(\CG)$ the set of all probability Borel regular $\CG$-invariant measures on $\GU$. In the case of transformation groupoid $\CG_{\Gamma\curvearrowright X}$ of a dynamical system $\Gamma\curvearrowright X$, we write $M_\Gamma(X)$ instead of $M(\CG_{\Gamma\curvearrowright X})$. The following is a groupoid analogue of the dynamical comparison (see \cite{D} for example).

\begin{defn}
	Let $\CG$ be a locally compact Hausdorff \'{e}tale groupoid. We say $\CG$ has \emph{groupoid comparison} if  $U \prec_{\CG} V$ holds for all open sets $U, V\subset \CG^{(0)}$ satisfying $\mu(U)<\mu(V)$ for all $\mu\in M(\CG)$. If a transformation groupoid $\CG_{\Gamma\curvearrowright X}$ has groupoid comparison, we say $\Gamma\curvearrowright X$ has \textit{dynamical comparison} instead.
\end{defn}

The idea of paradoxicality, dating back to the work of Hausdorff and playing an important role in the work of Banach-Tarski (see \cite{Wagon}), roughly speaking, is that one object somehow contains two disjoint copies of itself. The notions of this flavor have been observed as a key condition implying pure infiniteness of related $C^*$-algebras (see \cite{R-S} and \cite{M2} for example). We now interpret this philosophy in the setting of locally compact Hausdorff \'{e}tale groupoids. Let $\CG$ be such a groupoid.  For any non-empty open sets $U, V$ in $\GU$, we write $U\prec_{\CG, 2} V$ if for any compact set $F\subset U$ there are disjoint non-empty open sets $O_1, O_2\subset V$ such that $F\prec_{\CG} O_1$ and $F\prec_{\CG} O_2$.

\begin{defn}
	Let $\CG$ be a locally compact Hausdorff  \'{e}tale groupoid. We say $\CG$ is \textit{purely infinite} if  for any non-empty open sets $O_1, O_2$ in $\CG^{(0)}$ satisfying $O_1\subset r(\CG O_2)$, one has $O_1\prec_{\CG, 2} O_2$.
\end{defn}
 
\begin{defn}
	Let $\CG$ be a locally compact Hausdorff  \'{e}tale groupoid. We say $\CG$ has \textit{paradoxical comparison} if $O\prec_{\CG, 2} O$ for all non-empty open sets in $\CG^{(0)}$.  
\end{defn} 

Furthermore, we also need the following weak version of pure infiniteness. 

\begin{defn}
	Let $\CG$ be a locally compact Hausdorff \'{e}tale groupoid. We say $\CG$ is \textit{weakly purely infinite} if for any non-empty open sets $O_1, O_2$ in $\CG^{(0)}$ satisfying $O_1\subset r(\CG O_2)$, one has $O_1\prec_{\CG} O_2$.
\end{defn}

Note that Matui (Definition 4.9 in \cite{Matui2}) introduced a notion called pure infiniteness as well for locally compact Hausdorff \'{e}tale ample groupoids on the Cantor set. It is also straightforward to verify that our paradoxical comparison is a generalization of the pure infiniteness in Matui's sense. We will actually show in Section 5 that they are equivalent for locally compact Hausdorff \'{e}tale ample groupoids. In addition, Suzuki (Definition 3.3 in \cite{YS}) also introduced a notion of pure infiniteness for locally compact Hausdorff \'{e}tale  groupoids on compact spaces. His notion is equivalent to Matui's notion when the groupoid is minimal and ample. The following shows that when the groupoid $\CG$ is ample it suffices to consider compact open sets instead of general open sets in Definition 3.6 above.  

\begin{prop}
	Let $\CG$ be a locally compact Hausdorff \'{e}tale  ample groupoid. Then $\CG$ has paradoxical comparison if and only if $O\prec_{\CG, 2} O$ for any non-empty compact open set $O$ in $\GU$.
\end{prop}
\begin{proof}
	It suffices to show the ``if'' part.  Now, let $U$ be a non-empty open set in $\GU$  and $F\subset U$ be a compact set. Since $\CG$ is ample, there is a basis consisting of compact open sets for the topology on $\GU$. Then because $F$ is compact and $\GU$ is locally compact Hausdorff, there is a compact open set $O$ such that $F\subset O\subset U$. Now since $O\prec_{\CG, 2} O$ holds by assumption, there are two disjoint open sets $O_1\subset O\subset U$ and $O_2\subset O\subset U$ such that  $F\prec_{\CG} O_1$ and $F\prec_{\CG} O_2$. This establishes $U\prec_{\CG, 2} U$ and thus $\CG$ has paradoxical comparison.
\end{proof}

\begin{Rmk}
Compare to \cite{B-L} and \cite{Ra-Sims}, in the case that $\CG$ is ample, for a non-empty compact open set $O$, it can be verified $O\prec_{\CG, 2} O$ if and only if $O$ is $(2 ,1)$-paradoxical in the sense of Definition 4.5 in \cite{B-L}.
Then Proposition 3.8 shows that  our paradoxical comparison is equivalent to that all compact open sets in $\GU$ is $(2 ,1)$-paradoxical if $\CG$ is ample.
\end{Rmk}

Let $\CG$ be a locally compact Hausdorff \'{e}tale groupoid. Recall that we have assumed in Section 2 that $\GU$ is infinite. Then, we remark that $\CG$ has the groupoid comparison and $M(\CG)=\emptyset$ if and only if $U\prec_{\CG} V$ holds for any non-empty open sets $U, V$ in $\GU$. To see this, it suffices to show that if $U\prec_{\CG} V$ holds for any non-empty open sets $U, V$ in $\GU$ then $M(\CG)$ is empty. Indeed, since $\GU$ is Hausdorff, there are two disjoint non-empty open sets $O_1, O_2$ such that $\GU\prec_{\CG} O_i$ for $i=1, 2$. For $O_1$ and any compact set $F\subset \GU$, because $F\prec_{\CG} O_1$, there are open bisections $U_1, \dots, U_n$ such that $F\subset \bigcup_{j=1}^ns(U_j)$ and $\bigsqcup_{j=1}^nr(U_j)\subset O_1$. Now suppose $\mu\in M(\CG)$, one has
\[\mu(F)\leq \sum_{j=1}^n\mu(s(U_j))=\sum_{j=1}^n\mu(r(U_j))\leq \mu(O_1).\]
Because $\mu$ is regular, one actually has $1=\mu(\GU)\leq \mu(O_1)\leq 1$ and thus $\mu(O_1)=1$. Then the same method also shows that $\mu(O_2)=1$. This is a contradiction because $O_1$ and $O_2$ are disjoint. In addition, groupoid comparison in the case $M(\CG)=\emptyset$ implies that the unit space $\GU$ is \textit{perfect} in the sense that there is no isolated units. Indeed, it is not hard to observe the cardinality inequality $|F|\leq |O|$ for every compact set $F$ and non-empty open set $O$ satisfying $F\prec_{\CG} O$. Now, suppose there is an open set $O$ whose cardinality is one. Let $F$ be a compact set consisting exactly two units. Let $U$ be an open set such that $F\subset U$. Since  $\CG$ is assumed to have groupoid comparison and $M(\CG)=\emptyset$, one has $U\prec_{\CG} O$ and thus $F\prec_{\CG} O$. But this is a contradiction because $|F|=2>1=|O|$. Now we have the following preliminary result.
\begin{lem}
	Let $\CG$ be a locally compact Hausdorff \'{e}tale  groupoid. If $\CG$ has groupoid comparison and $M(\CG)=\emptyset$ then $\CG$ is minimal and purely infinite.
\end{lem}
\begin{proof}
Suppose that $\CG$ has groupoid comparison and $M(\CG)=\emptyset$. Then for any non-empty open sets $U, V$ in $\GU$ one has $U\prec_{\CG} V$. This shows that for any unit $u$ and any non-empty open set $O$ in $\GU$ one has $\{u\}\prec_{\CG} O$ because for any open neighborhood $W$ of $u$, one has $W\prec_{\CG} O$ by our assumption. Therefore, there is a $\gamma\in \CG$ such that $s(\gamma)=u$ and $r(\gamma)\in O$ and thus $\CG$ is minimal. Thus, $U\subset r(\CG V)=\GU$ holds trivially for any open sets $U, V$ in $\GU$. Now since $\GU$ is Hausdorff and perfect, one can choose two disjoint non-empty open sets $V_1, V_2\subset V$. Then the fact that $U\prec_{\CG} V_1$ and $U\prec_{\CG} V_2$ implies that $U\prec_{\CG, 2} V$, which shows that $\CG$ is purely infinite.
\end{proof}

To close this section, we list several  natural examples of locally compact Hausdorff \'{e}tale purely infinite groupoids. It is straightforward to see pure infiniteness implying paradoxical comparison. However, we remark in advance that pure infiniteness is actually equivalent to paradoxical comparison. This equivalence will be established in Theorem 5.1 below.

\begin{Exl}(Strong boundary actions  and $n$-filling actions)

For discrete group acting on Compact Hausdorff spaces, Laca and Spielberg in \cite{L-S}, introduced strong boundary actions. Motivated by their work, then in \cite{J-R}, Jolissaint and Robertson introduced the $n$-filling action. An action is a strong boundary action exactly when it is $2$-filling. It was proved in \cite{M2} that all $n$-filling actions, thus including the strong boundary actions, are examples of actions satisfying dynamical comparison but having no invariant probability Borel measures.  Therefore the transformation groupoid of a $n$-filling action is purely infinite by Lemma 3.10. This class includes actions of hyperbolic groups on their Gromov boundaries (Example 2.1 in \cite{L-S}), the canonical action of $SL_n(\Z)$ on the projective space $\mathbb{P}^{n-1}(\R)$ (Example 2.1 in \cite{J-R}) and $H_0(M)$ acting on $M$ introduced in \cite{YS}, where $M$ is a connected compact manifold with no boundaries and $H_0(M)$ is the path connected component of the group of all homeomorphism of $M$ containing the identity. See more in \cite{L-S}, \cite{J-R} and \cite{YS}. 
\end{Exl}

We remark that paradoxical comparison is preserved by inverse limit for dynamical systems. Let $\alpha_n: \Gamma\curvearrowright X_n$ for $n\in \N$ be a sequence of actions of a discrete group $\Gamma$ on compact Hausdorff space $X_n$. Let $\pi_{n}: X_{n+1}\to X_n$ be a factor map.  Then the \textit{inverse limit} system $\alpha: \Gamma\curvearrowright X$ is defined by
\[X=\{(x_n)\in \prod_{n\in \N}X_n: \pi_{n}(x_{n+1})=x_n\ \text{for any}\ n\in \N\}\] together with
$\gamma\cdot(x_n)=(\gamma x_n)$ for any $\gamma\in\Gamma$ and $(x_n)\in X$. Note that $X$ is equipped with the relative product topology inherited from $\prod_{n\in \N}X_n$ and it is not hard to see $X$ is compact as well. Denote by $P_n$ the canonical projection from $\prod_{n\in \N}X_n$ to $X_n$. Let $n<m\in \N$. Denote by $\pi_{n, m}=\pi_{n}\circ\pi_{n+1}\circ\dots\circ\pi_{m-1}$, which is a factor map from $X_m$ to $X_n$. Define $\pi_{n, n}=\id_n$, i.e., the identity map on $X_n$. For any $n\leq m$, observe that $\pi_{n, m}\circ P_m=P_n$ when restrict  $P_n$ and $P_m$ on $X$. Let $O\subset X_n$ be an open set in $X_n$. Denoted by $B(O)$ the open set $P_n^{-1}(O)\cap X$ for simplicity. The following is a preliminary result.

\begin{lem}
	Let $\alpha: \Gamma\curvearrowright X$ be the inverse limit system of $\alpha_n: \Gamma\curvearrowright X_n$ for $n\in \N$ mentioned above. Then the collection $\CC=\{B(O): O\ \text{is an open set in}\ X_n, n\in \N\}$ form a base of the topology on $X$.
\end{lem}
\begin{proof}
	Note that $\CC$ is a subbase for the topology on $X$. Then it suffices to verify for any finite set $\{n_1, \dots, n_k\}\subset \N$ and open sets $O_i\subset X_{n_i}$  there is an $n\in \N$ and a non-empty open set $U\subset X_n$ such that $B(U)\subset \bigcap_{i=1}^kB(O_i)$ whenever $\bigcap_{i=1}^kB(O_i)\neq \emptyset.$  Suppose $\bigcap_{i=1}^kB(O_i)\neq \emptyset$ holds. Define $n=\max\{n_i: i=1,\dots, k\}$ and $V_i=\pi^{-1}_{n_i, n}(O_i)$, which is an open set in $X_n$.  Then since $\pi_{n_i, n}\circ P_n=P_{n_i}$ on $X$, one has 
	\[B(V_i)=P^{-1}_{n}(V_i)\cap X=P^{-1}_{n}(\pi^{-1}_{n_i, n}(O_i))\cap X=P^{-1}_{n_i}(O_i)\cap X=B(O_i).\]
Then one also has $\bigcap_{i=1}^kV_i\neq \emptyset$. Now define
	$U=\bigcap_{i=1}^kV_i$. Then by the definition one has 
\[B(U)=P^{-1}_n(U)\cap X=\bigcap_{i=1}^kP^{-1}_n(V_i)\cap X=\bigcap_{i=1}^kB(O_i)\] as desired.
\end{proof}

Now we have the following permanence result.

\begin{prop}
Let $\alpha: \Gamma\curvearrowright X$ be the inverse limit system of $\alpha_n: \Gamma\curvearrowright X_n$ for $n\in \N$ mentioned above. If each $\alpha_n$ has paradoxical comparison then so is $\alpha$.
\end{prop}
\begin{proof}
	Let $U$ be an open set in $X$ and $F\subset U$ a compact set. Then Lemma 3.12 and compactness of $F$ imply that there is an $n\in \N$ and open sets $O_1, \dots, O_k$ in $X_n$ such that $F\subset \bigcup_{i=1}^kB(O_i)\subset U$. Define $K=P_n(F)$, which is a compact set in $X_n$ and $K\subset \bigcup_{i=1}^kO_i$. Write $O=\bigcup_{i=1}^kO_i$ for simplicity. Since $\alpha_n$ has paradoxical comparison, there are disjoint open sets $V_1, V_2\subset O$ such that $K\prec V_j$ for $j=1, 2$. For each $j=1, 2$, there is a collection $\{U^j_1, \dots, U^j_{k_j}\}$ of open sets in $X_n$ and a collection $\{\gamma^j_1, \dots, \gamma^j_{k_j}\}$ of group elements in $\Gamma$ such that $K\subset \bigcup_{i=1}^{k_j}U^j_i$ and $\bigsqcup_{i=1}^{k_j}\gamma^j_iU^j_i\subset V_j$. Since $X$ is an inverse limit, each restriction of $P_n$ on $X$ is a factor map. Then, for any $i\leq k_j$, one has $P_n^{-1}(\gamma^j_iU^j_i)\cap X=\gamma^j_i(P_n^{-1}(U^j_i)\cap X)$. This implies that $F\subset \bigcup_{i=1}^{k_j}(P_n^{-1}(U^j_i)\cap X)$ and $\bigsqcup_{i=1}^{k_j}\gamma^j_i(P_n^{-1}(U_i)\cap X)\subset P_n^{-1}(V_j)\cap X=B(V_j)$. This shows that $F\prec B(V_j)$ for each $j=1, 2$. Now recall $V_1, V_2$ are disjoint. Then so are $B(V_1)$ and $B(V_2)$. In addition, for each $j=1, 2$, one has 
	\[B(V_j)=P_n^{-1}(V_j)\cap X\subset P_n^{-1}(O)\cap X=\bigcup_{i=1}^k(P_n^{-1}(O_i)\cap X)\subset U.\]
	This shows that $\alpha$ has paradoxical comparison.
\end{proof}

\begin{Exl}
In \cite{YS1}, Suzuki constructed many examples of unital Kirchberg algebras by using the inverse limit of  actions of free groups $\F_n$ on its boundary. Note that such an action is a strong boundary action in the sense of \cite{L-S} and thus has paradoxical comparison by Lemma 3.10 (see also Example 3.11 above). Then Proposition 3.13 shows that  Suzuki's examples in \cite{YS1} have paradoxical comparison and thus are purely infinite.
\end{Exl}

\begin{Exl}($n$-filling locally compact Hausdorff \'{e}tale groupoids)
	
	In \cite{YS}, Suzuki generalized the $n$-filling actions mentioned above to locally compact Hausdorff \'{e}tale groupoids on compact spaces.  Such a groupoid $\CG$ is called $n$-filling if for any non-empty open set $W$ in $\GU$ there are $n$ open bisections $E_1, \dots, E_n$ such that 
	$\bigcup_{i=1}^nr(E_iW)=\GU$.  Rainone and Sims in \cite{Ra-Sims} provided another but equivalent generalization in the sense that for any $n$ open sets $W_1, \dots, W_n$ in $\GU$ there are $n$ open bisections $E_1, \dots, E_n$ such that $\bigcup_{i=1}^nr(E_iW_i)=\GU$. We remark that even Rainone and Sims proved this equivalence only in the ample case, their proof works in general. It is not hard to see that if a locally compact Hausdorff \'{e}tale groupoid $\CG$ on a compact space $\GU$ is $n$-filling then $\CG$ has groupoid comparison and $M(\CG)=\emptyset$. Indeed, since $\GU$ is assumed to be compact, it suffices to show $\GU\prec_{\CG} U$ for any open set $U$ in $\GU$. Choose $n$ non-empty disjoint open subsets $W_1, \dots, W_n$ of $U$ and there are open bisections $E_1, \dots, E_n$ such that $\bigcup_{i=1}^nr(E_iW_i)=\GU$. Then for each $i=1, \dots, n$, define open bisections $F_i=(s|_{E_i})^{-1}(W_i)$ satisfying $s(F_i)\subset W_i$ and $r(F_i)=r(E_iW_i)$. Then one has $\GU=\bigcup_{i=1}^ns(F^{-1}_i)$ and $\bigsqcup_{i=1}^nr(F_i^{-1})\subset \bigsqcup_{i=1}^nW_i\subset U$, which establishes the groupoid comparison. Finally, we warn that the notion of  $n$-filling in \cite{J-R} for dynamical systems does not coincide with the $n$-filling in the sense of Suzuki or Rainone-Sims for the transformation groupoid of the dynamical systems.  We will come back to this in Section 5.
\end{Exl}

\begin{Exl}
Anantharaman-Delaroche introduced locally contracting groupoid in \cite{A-D}. A locally compact Hausdorff \'{e}tale groupoid $\CG$ is called \textit{locally contracting} if for any non-empty open set $U$ in $\GU$, there exists an open subset $V$ of $U$ and an open bisection $O$ such that $\overline{V}\subset s(O)$ and $r(O\overline{V})\subsetneq V$. It is not clear in general whether locally contracting implies pure infiniteness for locally compact Hausdorff \'{e}tale groupoids. However, we will show below that many examples of locally contracting groupoid in fact are purely infinite. 
\end{Exl}

First, it was noted in \cite{A-D} that if a minimal dynamical system $\Gamma \curvearrowright X$ satisfies the condition that there is a group element $g\in\Gamma$ having a fixed point $x_0$ as an \textit{attractor} of $g$ in the sense that there is an open neighborhood $W$ of $x_0$ such that $\{g^n(W): n\in \N\}$ form a neighborhood base at $x_0$ then the system is locally contracting.  However, Jolissaint and Robertson shows  this condition in fact implies that the action $\Gamma \curvearrowright X$ is $n$-filling for some $n\in \N^+$ when the underlying space $X$ is compact. In the following proposition, we show that this condition actually implies dynamical comparison in the more general setting that $X$ is locally compact.
\begin{prop}
	Let $\alpha: \Gamma \curvearrowright X$ be a minimal action of a discrete group $\Gamma$ on a locally compact Hausdoff non-discrete space $X$. Suppose there is a group element $g\in\Gamma$ having a fixed point $x_0$ which is an attractor in the sense that there is an open neighborhood $W$ of $x_0$ such that $\{g^n(W): n\in \N\}$ form a neighborhood basis at $x_0$. Then $\alpha$ has dynamical comparison and $M_\Gamma(X)=\emptyset$.
\end{prop}
\begin{proof}
	First we claim $X$ is perfect. Suppose not, let $x$ be an isolated point in $X$. Since $\Gamma$ is discrete and $\alpha$ is minimal, the space $X=\Gamma\cdot x$ is discrete. This is a contradiction to our assumption on $X$.
	Let $U, V$ be non-empty open sets in $\GU$ and $F\subset U$ be a compact set.  First, since $\alpha$ is minimal, there are finitely many group elements $h_1,\dots, h_n\in \Gamma$ such that $F\subset \bigcup_{i=1}^nh_iW$. Now since $X$ is perfect and Hausdorff, choose $n$ disjoint non-empty open subsets $V_1, \dots, V_n$ of $V$. Apply minimality of $\alpha$ again, one has that for each $i\leq n$ there is an $s_i\in \Gamma$ such that $h_ix_0\in s_iV_i$. This implies that there is a $n_i\in \N$ such that $x_0\in g^{n_i}W\subset h_i^{-1}s_iV_i$. Now set $t_i=s_i^{-1}h_ig^{n_i}h_i^{-1}$ for each $i\leq n$. Observe that $t_ih_iW\subset V_i$ and thus one has $\bigsqcup_{i=1}^nt_ih_iW\subset V$ since all $V_i$ are disjoint. This thus establishes $F\prec V$. Since $F$ is arbitrary, one has $U\prec V$. This shows that $\alpha$ has dynamical comparison and $M_\Gamma(X)=\emptyset$.
\end{proof}

Therefore, all examples of dynamical systems in \cite{A-D}  satisfy dynamical comparison and have no invariant probability measures. See Proposition 3.2, 3.3, 3.4, 3.5 in \cite{A-D}. 

Independently, Laca and Spielberg in \cite{L-S} introduced a dynamical version of the locally contracting called \textit{local boundary action}. There are several examples of local boundary action presented in \cite{L-S}. For instance, some Ruelle algebras arising from certain Smale spaces, for example, solenoids,  could be realized as reduced crossed products of dynamical systems satisfying the hypothesis of the following corollary of Proposition 3.17 (see Lemma 10 and Example 14.1, 14.2 in \cite{L-S}).

\begin{cor}
Let $\Gamma$ be a non-discrete locally compact Hausdorff group and $\Lambda$ be a dense subgroup of $\Gamma$. Suppose $\alpha \in \operatorname{Aut}(\Gamma)$ satisfies that 
\begin{enumerate}[label=(\roman*)]
	\item $\alpha(\Lambda)=\Lambda$;
	
	\item there is an open neighbourhood $U$ of the identity $e\in \Gamma$ such that $\{\alpha^k(U): k\in \Z\}$ is a neighbourhood base at $e$ in $\Gamma$.
\end{enumerate}
Then the dynamical system $\Lambda\rtimes_\alpha \Z\curvearrowright \Gamma$ is minimal and has dynamical comparison with $M_{\Lambda\rtimes \Z}(\Gamma)=\emptyset$, where the semi-product group $\Lambda\rtimes_\alpha \Z$ is given with discrete topology.
\end{cor}

We now recall the following concept in \cite{L-S}. 

\begin{defn}
Let $g$ be a homeomorphism of a locally compact Hausdorff space $X$. A fixed point $x$ of $g$ is called \textit{stable} if for any neighborhood $U_1$ of $x$ there is another neighborhood $U_2$ of $x$ such that $U_2\subset U_1$ and $g^nU_2\subset U_1$ for any $n\in \N$. The fixed point $x$ of $g$ is called \textit{asymtotically stable} if it is stable and there is a neighborhood $U$ such that $\lim_{n\to \infty}g^ny=x
$ for any $y\in U$. 
\end{defn}

It was proved in \cite{L-S} that if the set of asymtotically stable fixed points is dense, which happens if the dynamical system is minimal, then the action is a local boundary action. We show below that an asymptotically stable point $x$ of a group element $g$ is actually an attractor of $g$.

\begin{prop}
Let $\alpha: \Gamma \curvearrowright X$ be an action of a discrete group $\Gamma$ on a locally compact Hausdoff space $X$. Then any asymptotically stable fixed point $x$ for an element $g\in \Gamma$ is an attractor of $g$.
\end{prop}
\begin{proof}
	First, without loss of generality, one can choose a compact neighborhood $U$ of $x$ such that 
	\[\lim\limits_{n\to \infty} g^ny=x\] 
	holds for any $y\in U$. Now	let $O$ be an open set containing $x$. Choose an open neighborhood $V$ of $x$ such that $g^nV\subset O$ for all $n\in \N$.   Then for each $y\in U$, since $\lim_{n\to \infty} g^ny=x$ there is an $n_y\in \N$ such that $g^{n_y}y\in V$. This implies that there is an open neighborhood $W_y$ of $y$ such that $g^{n_y}W_y\subset V$. Note that $\{W_y: y\in U\}$ form a cover of $U$ and thus there is a finite subcover $\{W_1, \dots, W_m\}$ since $U$ is compact. Now define $n=\max\{n_1,\dots, n_m\}$ in which we write $n_i$ for $n_{y_i}$ to simplify the notation. Then one has 
	\[g^nU\subset g^n(\bigcup_{i=1}^mW_i)=\bigcup_{i=1}^mg^nW_i=\bigcup_{i=1}^mg^{n-n_i}g^{n_i}W_i\subset \bigcup_{i=1}^mg^{n-n_i}V\subset O.\]
	This shows that $x$ is an attractor of $g$.
\end{proof}

\begin{cor}
	Let $\alpha: \Gamma \curvearrowright X$ be a minimal action of a discrete group $\Gamma$ on a locally compact Hausdoff space $X$. Suppose there is an asymptotic stable point $x$ for a $\gamma\in \Gamma$. Then $\alpha$ has dynamical comparison and $M_\Gamma(X)=\emptyset$.
\end{cor}
\begin{proof}
This is a straightforward consequence of Proposition 3.17 and 3.20.
\end{proof}

Note that the action of $PSL(n, \Z)$ on a flag manifold $\CF$ in $\R^n$ considered in \cite{L-S} is minimal and has an asymptotically stable fixed point by some group elements. To summarize, we have the following result. 

\begin{Exl}
	Corollary 3.18 and 3.21 imply that all examples of local boundary actions in \cite{L-S} satisfy dynamical comparison and have no invariant probability measures. This includes the induced dynamical systems of $m$-adic solenoids in Example 14.1 and the hyperbolic automorphisms of tori in Example 14.2 as well as the action of $PSL(n, \Z)$ on a flag manifold $\CF$ in $\R^n$ in Example 16.1 of \cite{L-S} 
\end{Exl}

\begin{Exl}
	It was noted in Example 4.4.7 of \cite{Rordam} that  Kumjian and Archbold independently observed that the Cuntz algebra $\CO_2$ has a dynamical model in the sense that there is an action $\alpha_0$ of $\Z_2\ast \Z_3$ on the Cantor set $X$ such that $\CO_2\simeq C(X)\rtimes_r (\Z_2\ast \Z_3)$ where $\alpha_0$ is defined as follows. Identify $X$ by $\{0, 1\}^\N$ and let $\varphi$ and $\psi$ be two homeomorphism on $X$ given by
	\[\xymatrix@C=0.4cm{
		 (0, x_2,x_3,\dots)\ar[rr]^{\varphi} && (1, x_2,x_3,\dots) \ar[rr]^{\varphi} && (0, x_2,x_3,\dots) }\]
	 and
	 \[\xymatrix@C=0.4cm{
	 	(0, x_2,x_3,\dots)\ar[rr]^{\psi} && (1, 1, x_2,x_3,\dots) \ar[rr]^{\psi} && (1, 0, x_2,x_3,\dots) \ar[rr]^{\psi} && (0, x_2,x_3,\dots)}.\]
 	Then $\varphi^2=\psi^3=\id_X$ and thus $\varphi$ and $\psi$ induces an action $\alpha_0$ on $X$. We show below that $\alpha_0$ has dynamical comparison and $M_{\Z_2\ast \Z_3}(X)=\emptyset$.
\end{Exl}

\begin{prop}
	The action $\alpha_0: \Z_2\ast \Z_3 \curvearrowright X$ above has dynamical comparison and $M_{\Z_2\ast \Z_3}(X)=\emptyset$.
\end{prop}
\begin{proof}
Since  $X=\{0, 1\}^\N$ is compact, it suffices to show $X\prec O$ for any open set $O$ in $X$. Note that the collection of all $N_{z_1z_2,\dots,z_n}=\{x\in X: x_i=z_i\ \text{for any}\ i\leq n\}$ where $z_1, z_2,\dots, z_n\in \{0, 1\}$ and $n\in \N$ form a standard base of the topology on $X$. Now $X=N_0\sqcup N_1$. In addition, choose two disjoint open sets $N_{z_1z_2,\dots,z_n}$ and $N_{y_1y_2,\dots, y_m } \subset O$. Without loss of generality, one can assume $n ,m\geq 2$. Now, it suffices to show that there are $g_1, g_2\in \Z_2\ast \Z_3$  such that $g_1N_0=N_{z_1z_2,\dots,z_n}$ and $g_2N_1=N_{y_1y_2,\dots, y_m }$.  For $N_{z_1z_2,\dots,z_n}$, where $n\geq 2$,  one has 
\begin{enumerate}[label=(\roman*)]
	\item if $z_1=z_2=1$ then $\psi^{-1}(N_{z_1z_2,\dots,z_n})=N_{0z_3,\dots, z_n}$.
	\item if $z_1=1$ and $z_2=0$ then $\psi(N_{z_1z_2,\dots,z_n})=N_{0z_3,\dots, z_n}$.
	\item if $z_1=0$ then $\varphi(N_{z_1z_2,\dots,z_n})=N_{1z_2,\dots,z_n}$
\end{enumerate}
This implies that there is a $g\in \Z_2\ast \Z_3$ such that $gN_{z_1z_2,\dots,z_n}=N_{z_2,\dots,z_n}$. Indeed, if $z_1=z_2=1$ define $g=\varphi\circ \psi^{-1}$. If $z_1=1$ and $z_2=0$ then define $g=\psi$. If $z_1=0$, by (iii) above, one can always reduce the problem to the case $z_1=1$ above.
Therefore, by induction there is an $h\in \Z_2\ast \Z_3$ such that $hN_{z_1z_2,\dots,z_n}=N_{z_n}$. If $z_n=0$ we are done and if $z_n=1$ then $\varphi(hN_{z_1z_2,\dots,z_n})=N_0$. This thus shows that there is a $g_1\in  \Z_2\ast \Z_3$ such that $g_1N_0=N_{z_1z_2,\dots,z_n}$. The same method shows that there is an $h_2$ such that $h_2N_0=N_{y_1y_2,\dots, y_m}$. Then define $g_2=h_2\circ \varphi$. Then $g_2N_1=N_{y_1y_2,\dots, y_m}$ as desired.
\end{proof}

\begin{Exl}
	It was noted in \cite{B-L} that all compact open sets in the following locally compact Hausdorff \'{e}tale ample groupoids are $(2,1)$-paradoxical.
	\begin{enumerate}[label=(\roman*)]
		\item Cuntz groupoids defined in \cite{Renault} ;
		\item  coarse groupoids generated by paradoxical coarse metric spaces with bounded geometry.
	\end{enumerate}
	 Remark 3.10  then implies the ample groupoids above have paradoxical comparison and thus are purely infinite.  
\end{Exl}

The following examples of purely infinite groupoids arise from directed graphs.

\begin{Exl}
As we mentioned above, our pure infiniteness is a generalization of Matui's pure infiniteness in \cite{Matui2}. See Corollary 5.5 below. It was also proved in \cite{Matui2} that the groupoids arising from shifts of finite type are purely infinite.
\end{Exl}

\begin{Exl}
 In \cite{Sp}, Spielberg constructed a groupoid model for each Kirchberg algebra in the UCT class. Each of these groupoid arises from a mixture of a $1$-graph and a $2$-graph. It can be verified that any of these groupoids is purely infinite by a virtually identical approach to Matui's argument for shifts of finite type. Therefore, any Kirchberg algebra in the UCT class has a purely infinite groupoid model. 
\end{Exl}

The following example on the negative side was communicated to the author by Hanfeng Li. This shows that paradoxical comparison for dynamical systems is not preserved by extensions.

\begin{Exl}(Hanfeng Li)
	Let $\alpha: \Gamma \curvearrowright X$ be an action of a countable discrete group $\Gamma$ acting on a compact metrizable space $X$ such that $\alpha$ has paradoxical comparison. Let $\Gamma^*=\Gamma\cup\{\infty\}$ be the one-point compactification  of $\Gamma$ and let $\beta: \Gamma \curvearrowright \Gamma^*$ be the action defined by $\beta_g(h)=gh$ if $g, h\in \Gamma$ and $\infty\in \Gamma^*$ is a fixed point by all $g\in \Gamma$.  Then $\alpha\times \beta: \Gamma \curvearrowright X\times \Gamma^*$ is an extension of $\alpha: \Gamma \curvearrowright X$. Now we show $\alpha\times \beta$ has no paradoxical comparison. Suppose the contrary, for any clopen set $X\times \{h\}$ for $h\in \Gamma$, there are disjoint non-empty open set $U_1, U_2\subset X\times \{h\}$ such that $X\times \{h\}\prec U_i$ for $i=1, 2$. Note that each $U_i$ is of the form $V_i\times \{h\}$ for some non-empty open set $V_i\subset X$. Then observe that the only the identity $e_\Gamma$ could implement $X\times\{h\}\prec V_1\times \{h\}$. But this implies that $V_1=X$ and thus $V_2=\emptyset$. But this is a contradiction to the assumption that $V_2$ is not empty.
\end{Exl}

\section{The groupoid semigroup}
In this section, we introduce a new semigroup for locally compact Hausdorff \'{e}tale groupoids as a groupoid version of the generalized type semigroup introduced in \cite{M3}. The main contribution here is to enlarge the definition of the semigroup to the case that the underlying space is locally compact Hausdorff. On the other hand, unlike the zero-dimensional case, the generalized type semigroup does not actually clearly reflect the ``type'' of decomposition of open sets in the underlying space. Therefore, we abandon the name ``the generalized type semigroup'' for groupoids here and simply call it the \textit{groupoid semigroup}. Before introducing the definition of groupoid semigroups, we recall some necessary backgrounds on preordered commutative semigroups. Recall a semigroup equipped with a neutral element is called a \textit{monoid}.

Let $(W, +, \leq)$ be a preordered commutative semigroup. We say  an element $x\in W$ is \textit{properly infinite} if $2x\leq x$. We say $W$  is \textit{purely infinite} if every $x\in W$ is properly infinite. In addition, we say $W$ is \textit{almost unperforated} if, whenever $x,y\in W$ and $n\in \mathbb{N}$ are such that $(n+1)x\leq ny$, one has $x\leq y$. A state on a preordered monoid $(W, +, \leq)$ is an order preserving morphism $f: W\rightarrow [0,\infty]$ with $f(0)=0$.  We say a state is \textit{non-trivial} if it takes a value different from $0$ and $\infty$. We denote by $S(W)$ the set consisting of all states of $W$ and by $S_N(W)$ the set of all non-trivial states. We write $S(W, x)=\{f\in S(W): f(x)=1\}$, which is a subset of $S_N(W)$. The following proposition due to Ortega, Perera, and R{\o}rdam is very useful.
See Proposition 2.1 in \cite{O-P-R}.

\begin{prop}
	Let $(W,+,\leq)$ be an ordered commutative semigroup, and let $x,y\in W$. Then the following conditions are equivalent:
	\begin{enumerate}[label=(\roman*)]
		\item There exists $k\in \mathbb{N}$ such that $(k+1)x\leq ky$.
		
		\item There exists $k_0\in \mathbb{N}$ such that $(k+1)x\leq ky$ for every $k\geq k_0$.
		
		\item There exists $m\in \mathbb{N}$ such that $x\leq my$ and $D(x)<D(y)$ for every state $D\in S(W,y)$.
	\end{enumerate}
\end{prop}

The following is a direct application of Proposition 4.1.

\begin{prop}
	Let $(W, +, \leq)$ be a preordered commutative monoid. Then $W$ is purely infinite if and only if $W$ is almost unperforated and $S_N(W)=\emptyset$. 
\end{prop}
\begin{proof}
	First suppose $W$ is purely infinite. Then for any $y\in W$, by induction, one has $my\leq y$ for any $m\in \N^+$. Now let $x, y\in W$. Suppose there is an $n\in \N^+$ such that $(n+1)x\leq ny$. Then one has $x\leq (n+1)x\leq ny\leq y$, which shows that $W$ is almost unperforated. Suppose there is a non-trivial state $f$ on $W$ such that $0<f(x)<\infty$ for some $x\in W$. Then pure infiniteness implies that $2x\leq x$ and thus one has $f(2x)=2f(x)\leq f(x)$, which is a contradiction to the fact $0<f(x)<\infty$. 
	
	For the reverse direction, let $x\in W$. Since there is no non-trivial state on $W$, in particular, one has $S(W, x)=\emptyset$. Then because $2x\leq mx$ for $m\geq 2$, Proposition 4.1 implies that there is an $k\in \N^+$ such that $(k+1)\cdot2x\leq kx$. Then since $W$ is almost unperforated, one has $2x\leq x$, which shows that $W$ is purely infinite.
\end{proof}

Now for a locally compact Hausdorff \'{e}tale groupoid $\CG$, we introduce the following subequivalence relation, which is the groupoid analogue of Definition 2.1 in \cite{M3}. However, we use compact sets instead of closed sets because our locally compact setting. 

\begin{defn}
	Let $\CG$ be a locally compact Hausdorff \'{e}tale groupoid. Let  $O_1,\dots, O_n$ and $V_1,\dots, V_m$ be two sequences of open sets in $\GU$, We write \[\bigsqcup_{i=1}^n O_i\times \{i\}\prec_{\CG} \bigsqcup_{l=1}^m V_l\times \{l\}\]
	if for every $i\in\{1,\dots, n\}$ and every compact set $F_i\subset O_i$ there are a collection of open bisections $W^{(i)}_1,\dots, W^{(i)}_{J_i}$ in $\CG$ and $k_1^{(i)},\dots, k_{J_i}^{(i)}\in \{1,\dots, m\}$ such that
	$F_i\subset \bigcup_{j=1}^{J_i}s(W^{(i)}_j)$ and
	\[\bigsqcup_{i=1}^n\bigsqcup_{j=1}^{J_i}r(W^{(i)}_j)\times \{k^{(i)}_j\}\subset \bigsqcup_{l=1}^m V_l\times \{l\}.\]
	\end{defn}

 We (logically and harmlessly) allow the empty set $\emptyset$ to appear as one or more of the open sets in Definition 4.3. In fact we make
\[\emptyset\prec_{\CG} \bigsqcup_{l=1}^m V_l\times \{l\}\]
sense for any $m\in \N^+$ and open sets $V_l$ by using bisections $W^{(i)}_j= \emptyset$. Note that the empty set $\emptyset$ above could be interpreted as $\bigsqcup_{i=1}^n \emptyset\times \{i\}$ for any $n\in \N^+$. We also emphasize that each $V_l$ above could also be the empty set. Denote by $\CO(\CG)$ the collection of all open sets in $\GU$ and write $\CO_n(\GU)$ for the set $\{(O_1,\dots, O_n): O_i\ \text{is open in }\GU\ \text{for any } i\leq n\}$.

\begin{defn}
	Let $\CG$ be a locally compact Hausdorff \'{e}tale groupoid. Let $a=(O_1,\dots, O_n)\in \CO_n(\GU)$ and $b=(U_1,\dots, U_m)\in \CO_m(\GU)$. We write $a\preccurlyeq b$ if \[\bigsqcup_{i=1}^n O_i\times \{i\}\prec_{\CG} \bigsqcup_{l=1}^m U_l\times \{l\}\] holds in the sense of Definition 4.3. 
\end{defn}

Write $\CK(\CG)=\bigcup_{n=1}^\infty \CO_n(\GU)$ and observe that the relation described in Definition 4.4 is in fact defined on $\CK(\CG)$. The following shows that the relation``$\preccurlyeq$'' is transitive. Let $a=(A_1, \dots, A_n)$ and $b=(B_1, \dots, B_m)$. In addition, let $a'=(A_{i_1}, \dots, A_{i_k})$ and $b'=(B_{j_1}, \dots, B_{j_l})$ be obtained by deleting all empty sets in the sequence $a$ and $b$, respectively. It is not hard to see both $a\preccurlyeq a'$ and $a'\preccurlyeq a$ hold. Furthermore, observe that $a\preccurlyeq b$ if and only if $a'\preccurlyeq b'$. This implies that  it suffices to consider all elements in $\CK(\CG)$ consisting of non-empty open sets. The proof of the following result is thus virtually identical to the proof of Lemma 2.2 in \cite{M3} by using bisections  in our general groupoid setting instead of group elements. In addition, since $\GU$ is not necessarily compact but locally compact, one needs to replace ``closed sets'' in the proof of Lemma 2.2 in \cite{M3} by ``compact sets''.  We thus omit the proof here.

\begin{lem}
Let $\CG$ be a locally compact Hausdorff \'{e}tale groupoid. Let $a, b, c\in \CK(\CG)$ such that $a\preccurlyeq b$ and $b\preccurlyeq c$. Then $a\preccurlyeq c$.
\end{lem}

Now define a relation on $\CK(\CG)$ by setting $a\approx b$ if $a\preccurlyeq b$  and $b\preccurlyeq a$ for $a,b\in \CK(\CG)$. To see that this relation is in fact an equivalence relation, first it is not hard to verify directly that $a\approx a$ for all $a\in \CK(\CG)$. In addition, by the definition of the relation ``$\approx$'',  one has that $a\approx b$ implies $b\approx a$ trivially. Now suppose $a\approx b$ and $b\approx c$. By definition one has $a\preccurlyeq b\preccurlyeq c$ and $c\preccurlyeq b\preccurlyeq a$. Then Lemma 4.5 entails that $a\preccurlyeq c$ and $c\preccurlyeq a$. This establishes $a\approx c$.

We write $\CW(\CG)$ for the quotient $\CK(\CG)/\approx$ and define an operation ``$+$'' on $\CW(\CG)$ by $[a]+[b]=[(a,b)]$, where $(a,b)$ is defined to be the concatenation of $a=(A_1,\dots, A_n)$ and $b=(B_1,\dots, B_m)$, i.e., $(a,b)=(A_1,\dots, A_n,B_1,\dots, B_m)$. It is not hard to see that if $a_1\preccurlyeq a_2$ and $b_1\preccurlyeq b_2$ then $(a_1, b_1)\preccurlyeq (a_2, b_2)$. This implies the operation ``$+$'' is well-defined and it can be additionally verified that the operation ``$+$''  is commutative, i.e, $[a]+[b]=[b]+[a]$.  Moreover, we endow $\CW(\CG)$ with the natural order by declaring $[a]\leq [b]$ if $a\preccurlyeq b$. Thus $\CW(\CG)$ is a well-defined commutative partially ordered monoid with the neutral element $0_{\CW(\CG)}=[(\emptyset)]$. Let $a=(A_1,\dots, A_n), b=(B_1,\dots, B_n)\in\CK(\CG)$ be two elements with same length and each of $A_i$ is precompact. We denoted by $a\leq_\CG b$ if $A_i\subset \overline{A_i}\subset B_i$ for each $i\leq n$. It is not hard to see that $a\leq_\CG b$ implies $a\preccurlyeq b$ trivially. 

The following is a groupoid analogue of a well-known result on the Cuntz semigroup (for example, see Proposition 2.17 in \cite{A-P-T}).

\begin{prop}
Let $a=(A_1, \dots, A_n), b=(B_1,\dots, B_m)$ be in $\CK(\CG)$, Then the following are equivalent.
\begin{enumerate}[label=(\roman*)]
	\item $a\preccurlyeq b$;
	
	\item $c\preccurlyeq b$ for any $c\in \CK(\CG)$ with $c\leq_\CG a$;
	
	\item for any $c\in \CK(\CG)$ with $c\leq_\CG a$ there is a $d\in\CK(\CG)$ with $d\leq_\CG b$ such that $c\preccurlyeq d$.
\end{enumerate}
\end{prop}
\begin{proof}
	(i)$\Rightarrow$(ii). Straightforward from the definitions of $c\leq_\CG a$ and $a\preccurlyeq b$.
	
	(ii)$\Rightarrow$(i). First since $\GU$ is locally compact Hausdorff, for any $1\leq i\leq n$ and compact set $F_i\subset A_i$ there is a precompact open set $V_i$ such that $F_i\subset V_i\subset \overline{V_i}\subset A_i$. Denote by $c=(V_1, \dots, V_n)\in \CK(\CG)$, which satisfies $c\leq_\CG a$.  Then (ii) implies that $c\preccurlyeq b$, which means there are a collection of open bisections $\{U^{(i)}_1,\dots,
	U^{(i)}_{J_i}\}$ and  $k_1^{(i)},\dots,
	k_{J_i}^{(i)}\in \{1,\dots, m\}$ such that $F_i\subset \bigcup_{j=1}^{J_i}s(U^{(i)}_j)$ and
	\[\bigsqcup_{i=1}^n\bigsqcup_{j=1}^{J_i}r(U^{(i)}_j)\times \{k^{(i)}_j\}\subset \bigsqcup_{l=1}^m B_l\times \{l\}.\]
	But this implies $a\preccurlyeq b$.
	
	(iii)$\Rightarrow$(ii). Now, suppose that (iii) holds. Then for every $c\leq_\CG a$, there is a $d\leq_\CG b$ such that $c\preccurlyeq d$. Thus, one has $c\preccurlyeq d\preccurlyeq b$ and thus $c\preccurlyeq b$ by Lemma 4.5. 
	
	(i)$\Rightarrow$(iii). Let $c=(C_1, \dots, C_n)\leq_\CG a$. Then each open set $C_i$ is precompact and satisfies $\overline{C_i}\subset A_i$. Now since $a\preccurlyeq b$ and $\GU$ is locally compact Hausdorff, there are a collection $\{U^{(i)}_1,\dots, U^{(i)}_{J_i}\}$ of precompact open bisections and  $k_1^{(i)},\dots, k_{J_i}^{(i)}\in
	\{1,\dots, m\}$ such that $\overline{C_i}\subset \bigcup_{j=1}^{J_i}s(U^{(i)}_j)$ and
	\[\bigsqcup_{i=1}^n\bigsqcup_{j=1}^{J_i}r(\overline{U^{(i)}_j})\times \{k^{(i)}_j\}\subset \bigsqcup_{l=1}^m B_l\times \{l\}.\]

	 Define $\mathcal{D}_{l}=\{r(\overline{U^{(i)}_j}): j=1,\dots, J_i, i=1,\dots, n, k^{(i)}_j=l\}$ and write $K_l=\bigsqcup \mathcal{D}_{l}$, which is a compact subset of $B_l$. Then there is a precompact open set $V_l$ such that $K_l\subset V_l\subset \overline{V_l}\subset B_l$. Define $d=(V_1, \dots, V_m)$, which satisfies $d\leq_\CG b$ and $c\preccurlyeq d$ as desired.
	
\end{proof}

\begin{defn}
	A state $D$ on the semigroup $\CW(\CG)$ is called \textit{lower semi-continuous} if $D([a])=\sup\{D([b]): b\in \CK(\CG), b\leq_\CG a\}$ for all $a\in \CK(\CG)$.
\end{defn}

For every state $D\in S(\CW(\CG))$, define $\bar{D}([a])=\sup\{D([b]): b\in \CK(\CG), b\leq_\CG a\}$. The following result shows that $\bar{D}$ is always a lower semi-continuous state on $\CW(\CG)$.

\begin{prop}
	For each state $D\in S(\CW(\CG))$, the induced function $\bar{D}$ above is a lower semi-continuous state.
\end{prop}
\begin{proof}
	Let $a\preccurlyeq b$. Then Proposition 4.6 implies that for any $c\leq_\CG a$ there is a $d\leq_\CG b$ such that $c\preccurlyeq d$. Then by the definition of $\bar{D}$, one has $D([c])\leq D([d])\leq \bar{D}([b])$ and thus $\bar{D}([a])\leq \bar{D}([b])$. This shows that $\bar{D}$ is monotone.

	Let $a, b\in \CK(\CG)$. If $\bar{D}([a])$ or $\bar{D}([b])$ is infinite then $\bar{D}([a]+[b])=\bar{D}([a])+\bar{D}([b])$ holds trivially since $\bar{D}$ is monotone. We then assume that both of them are finite. For any $c\leq_\CG a$ and $d\leq_\CG b$ it is not hard to see $(c,d)\leq_\CG (a, b)$. Then one has 
	\[D([c])+D([d])=D([(c,d)])\leq \bar{D}([(a,b)])=\bar{D}([a]+[b]),\] which implies that $\bar{D}([a])+\bar{D}([b])\leq \bar{D}([a]+[b]).$ To show the reverse direction, first write $a=(A_1, \dots, A_n)$ and $b=(B_1,\dots, B_m)$. Then any $c\leq_\CG (a, b)$ has the form $c=(C_1, \dots, C_n, D_1, \dots, D_m)$, where sets $C_i$ and $D_j$ are precompact open such that $C_i\subset \overline{C_i}\subset A_i$ and  $D_j\subset \overline{D_j}\subset B_j$  for all $1\leq i\leq n$ and $1\leq j\leq m$. Define $d_1=(C_1, \dots, C_n)\leq_\CG a$ and $d_2=(D_1, \dots, D_m)\leq_\CG b$. Note that $[c]=[d_1]+[d_2]$. Then one has 
	\[D([c])=D([d_1])+D([d_2])\leq \bar{D}([a])+\bar{D}([b]),\] which implies that $\bar{D}([a]+[b])\leq \bar{D}([a])+\bar{D}([b])$ and thus in fact one has \[\bar{D}([a]+[b])= \bar{D}([a])+\bar{D}([b]).\]
	This verifies that $\bar{D}$ is a state.
	
	For lower semi-continuity, suppose $a=(A_1, \dots, A_n)$ and $b=(B_1,\dots, B_n)\in \CK(\CG)$ with $b\leq_\CG a$.  Then for each $i\leq n$ one has $B_i\subset \overline{B_i}\subset A_i$ where $B_i$ is precompact.  Since $\GU$ is locally compact Hausdorff, for each $i\leq n$ there is a precompact open set $U_i$ such that $\overline{B_i}\subset U_i\subset \overline{U_i}\subset A_i$. Define $c=(U_1, \dots, U_n)$, which satisfies $b\leq_\CG c\leq_\CG a$. Then one has 
	\[D([b])\leq \bar{D}([c])\leq \sup\{\bar{D}([d]): d\in \CK(\CG), d\leq_\CG a\}\leq \bar{D}([a]),\] and thus by the definition of $\bar{D}$, one has
	\[\bar{D}([a])=\sup\{\bar{D}([c]): c\in \CK(\CG), c\leq_\CG a\}.\] This establishes lower semi-continuity of $\bar{D}$.
\end{proof}

 We say a function $\mu: \CO(\GU)\rightarrow [0, \infty]$ is a \emph{groupoid dimension function} if $\mu$ satisfies the following property.
\begin{enumerate}[label=(\roman*)]
	\item $\mu(\emptyset)=0$;
	
	\item $\mu(s(V))=\mu(r(V))$ for any open bisection $V$;
	
	\item $\mu(O_1)\leq \mu(O_2)$ if $O_1\subset O_2$;
	
	\item if $O_1, O_2$ are open sets in $\GU$ then $\mu(O_1\cup O_2)\leq\mu(O_1)+\mu(O_2)$. If $O_1$ and $O_2$ are disjoint then $\mu(O_1\sqcup O_2)=\mu(O_1)+\mu(O_2)$
\end{enumerate}
We say a groupoid dimension function $\mu$ is \textit{non-trivial} if there is an open set $O$ such that $0<\mu(O)<\infty$. 
There is a natural way to extend the definition of the groupoid dimension function to compact sets. Let $\mu$ be a groupoid dimension function. For any compact set $K$ in $\GU$, define $\mu(K)=\inf\{\mu(O): K\subset O, O\ \textrm{open}\}$. We call such an extension the \emph{outer regular} extension of the groupoid dimension function $\mu$, denoted by $\mu$ as well.  Let $\mu$ be an outer regular extended groupoid dimension function. It is not hard to see that for any open set $O$ and a compact set $K$ if  $O\subset K$ then $\mu(O)\leq \mu(K)$ by the definition of $\mu(K)$ and property (iii) of the groupoid dimension function $\mu$.

Note that if $K\prec_{\CG} O$ for some compact set $K$ and open set $O$ then $\mu(K)\leq \mu(O)$ for any outer regular extended groupoid dimension function $\mu$. Indeed, Let $K\prec_{\CG} O$. Then there are open bisections $V_1, \dots, V_n$ such that $K\subset\bigcup_{i=1}^ns(V_i)$ and $\bigsqcup_{i=1}^n r(V_i)\subset O$. Then for any outer extended groupoid dimension function $\mu$, one has 
\[\mu(K)\leq \mu(\bigcup_{i=1}^ns(V_i))\leq \sum_{i=1}^n\mu(s(V_i))=\sum_{i=1}^n\mu(r(V_i))= \mu(\bigsqcup_{i=1}^n r(V_i))\leq \mu(O).\] 

We say an outer regular extended groupoid dimension function $\mu$ is \emph{regular} if $\mu(O)=\sup\{\mu(K): K\subset O, K\
\textrm{compact}\}$ for any open set $O$ in $\GU$.  Denote by $\CD_R(\CG)$ the set of all regular groupoid dimension functions on $\GU$. Note that by regularity, if $O_1\prec_{\CG} O_2$ for non-empty open sets $O_1, O_2$ in $\GU$ then $\mu(O_1)\leq \mu(O_2)$ for any $\mu\in \CD_R(\CG)$.

\begin{prop}
	Let $\CG$ be a locally compact Hausdorff groupoid. Then every state $D\in S(\CW(\CG))$ induces a groupoid dimension function $\mu_D$. If $D$ is non-trivial then $\mu_D$ is non-trivial. In addition, if $D$ is lower semi-continuous then the outer regular extension of $\mu_D$ is regular.
\end{prop}
\begin{proof}
	For every open set $O$, define $\mu_D(O)=D([a])$ where $a=(O)$. First $\mu_D(\emptyset)=D(0_{\CW(\CG)})=0$. Let $V$ be an open bisection in $\CG$. Then one has $s(V)\prec_{\CG} r(V)$ as well as $r(V)\prec_{\CG} s(V)$. This implies that $(s(V))\approx (r(V))$ in $\CK(\CG)$ and thus one has $\mu_D(s(V))=\mu_D(r(V))$. In addition, if $O_1\subset O_2$ then $(O_1)\preccurlyeq(O_2)$ holds automatically in $\CK(\CG)$  and thus $\mu_D(O_1)\leq \mu_D(O_2)$.
	Then for any open sets $O_1, \dots, O_n$ in $\GU$, it is not hard to observe that 
	\[(\bigcup_{i=1}^nO_i)\preccurlyeq (O_1, \dots, O_n)\] in $\CK(\CG)$, which implies that 
	\[\mu_D(\bigcup_{i=1}^n O_i)=D([(\bigcup_{i=1}^n O_i)])\leq D([(O_1, \dots, O_n)])=\sum_{i=1}^n D([(O_i)])=\sum_{i=1}^n \mu_D(O_i).\]
	Moreover, if $O_1,\dots O_n$ are pairwise disjoint then one has 	\[(\bigsqcup_{i=1}^nO_i)\approx (O_1, \dots, O_n)\] in $\CK(\CG)$. Then one has the additivity:
	\[\mu_D(\bigsqcup_{i=1}^n O_i)= \sum_{i=1}^n \mu_D(O_i).\]
	This shows that $\mu_D$ is a groupoid dimension function.
	
	Suppose $D$ is non-trivial. Then there is an $a=(A_1, \dots, A_n)\in \CK(\CG)$ such that $0<D([a])<\infty$. Since $D([a])=\sum_{i=1}^nD([(A_i)])$,  there exists at least one $i\leq n$ such that $0<D([(A_i)])<\infty$ and thus one has $0<\mu_D(A_i)<\infty$, which shows that $\mu_D$ is non-trivial.
	
	Now suppose $D$ is lower semi-continuous. For any compact set $F$, define $\mu_D(F)=\inf\{\mu_D(O): F\subset O, O\ \textrm{open}\}$, which provides the outer regular extension of $\mu_D$. Since $D$ is lower semi-continuous, one has 
	\begin{align*}
	\mu_D(O)=D([(O)])&=\sup\{D([(V)]): V\subset \overline{V}\subset O, \overline{V}\ \textrm{compact}\}\\
	&=\sup\{\mu_D(V): V\subset \overline{V}\subset O, \overline{V}\ \textrm{compact}\}.
	\end{align*}
	Then combining the fact $\mu_D(V)\leq \mu_D(\overline{V})\leq \mu_D(O)$, one has 
	\[\mu_D(O)=\sup\{\mu_D(F): F\subset O, F\ \textrm{compact}\},\] which shows that $\mu_D$ is regular.
\end{proof}

For the reverse direction,  every regular groupoid dimension function $\mu$ on $\GU$ induces a map $D_\mu: \CW(\CG) \to [0, \infty]$ by $D_\mu([a])=\sum_{i=1}^n\mu(A_i)$ for $a=(A_1,\dots, A_n)\in \CK(\CG)$. We will show below $D_\mu$ is a lower semi-continuous state on $\CW(\CG)$.

\begin{prop}
	For each regular groupoid dimension function $\mu$, the map $D_\mu$ defined above is a lower semi-continuous state on $\CW(\CG)$.
\end{prop}
\begin{proof}
	 Suppose $a=(A_1, \dots, A_n)$ and $b=(B_1,\dots, B_m)\in \CK(\CG)$ such that $a\preccurlyeq b$.  Let $c=(C_1,\dots, C_n)\leq_\CG a$.  Since $\overline{C_i}$ is a compact subset of $A_i$ for all $i\leq n$,  there are a collection $\{U^{(i)}_1,\dots, U^{(i)}_{J_i}\}$ of precompact open bisections and  $k_1^{(i)},\dots, k_{J_i}^{(i)}\in
	\{1,\dots, m\}$ such that $\overline{C_i}\subset \bigcup_{j=1}^{J_i}s(U^{(i)}_j)$ and
	\[\bigsqcup_{i=1}^n\bigsqcup_{j=1}^{J_i}r(U^{(i)}_j)\times \{k^{(i)}_j\}\subset \bigsqcup_{l=1}^m B_l\times \{l\}.\]
	Let $\mu$ be a regular groupoid dimension function. Then one has 
	\[\sum_{i=1}^n\mu(\overline{C_i})\leq\sum_{i=1}^n\sum_{j=1}^{J_i}\mu(s(U^{(i)}_j))=\sum_{i=1}^n\sum_{j=1}^{J_i}\mu(r(U^{(i)}_j))\leq \sum_{l=1}^m\mu(B_l).\]  Then since $\mu$ is regular, one has $$\sum_{i=1}^n\mu(A_i)\leq \sum_{l=1}^m\mu(B_l).$$
	This shows that $D_\mu$ is an order preserving map and also well-defined on $\CW(\CG)$. Then the additivity of $D_\mu$ is clear from the definition of $D_\mu$ above and $D_\mu([\emptyset])=0$.
	To show the lower semi-continuity of $D_\mu$, for any $i\leq n$  and a compact set $F_i\subset A_i$ there is a $c=(C_1,\dots, C_n)\leq_\CG a$ such that $F_i\subset C_i\subset \overline{C_i}\subset A_i$. Then one has 
	\[\sum_{i=1}^n\mu(F_i)\leq \sum_{i=1}^n\mu(C_i)\leq \sum_{i=1}^n\mu(A_i).\] Now because $\mu$ is regular, one has
	\[\sum_{i=1}^n\mu(A_i)=\sup\{\sum_{i=1}^n\mu(C_i): c=(C_1,\dots, C_n)\leq_\CG a\}.\]
	This shows that $D_\mu([a])=\sup\{D_\mu([c]): c\in \CK(\CG), c\leq_\CG a\}$ and thus $D_\mu$ is lower semi-continuous.
\end{proof}

We denote by $\operatorname{Lsc}(\CW(\CG))$ the set of all lower semi-continuous states on $\CW(\CG)$.

\begin{thm}
	The map  $S: \operatorname{Lsc}(\CW(\CG))\rightarrow \CD_R(\CG)$ defined by $S(D)=\mu_D$ is an affine bijection.
\end{thm}
\begin{proof}
	Proposition 4.9 implies that the map $S: D\mapsto\mu_D$ is well-defined. Then we show that $S$ is injective. If $\mu_{D_1}=\mu_{D_2}$ then for any open set $O$ in $\GU$ one has 
	\[D_1([(O)])=\mu_{D_1}(O)=\mu_{D_2}(O)=D_2([(O)]).\] 
	Then for any $a=(O_1, \dots, O_n)$ one has 
	\[D_1([a])=\sum_{i=1}^n D_1([(O_i)])=\sum_{i=1}^n D_2([(O_i)])=D_2([a]).\] To see the surjectivity of $S$ it suffices to observe that $S(D_\mu)=\mu$ for any $\mu\in \CD_R(\CG)$ by Proposition 4.10. 
	
	Finally, let $D=\lambda D_1+(1-\lambda)D_2$, where $0\leq \lambda\leq 1$ and $D, D_1, D_2\in \operatorname{Lsc}(\CW(\CG))$. Then for any open set $O$ in $\GU$, one has 
	\[\mu_D(O)=D([(O)])=\lambda D_1([(O)])+(1-\lambda)D_2([(O)])=\lambda \mu_{D_1}(O)+(1-\lambda)\mu_{D_2}(O).\]
	Then the regularity of $\mu_D, \mu_{D_1}$ and $\mu_{D_2}$ implies that $S$ is affine.
\end{proof}

Then using Proposition 4.1, 4.6, Theorem 4.11.  we are able to characterize when $\CW(\CG)$ is almost unperforated, which is a groupoid version of Theorem 3.9 in \cite{M3}. This can be done via a straightforward generalization of Lemma 3.8 and Theorem 3.9 in \cite{M3}. Therefore, we still omit the proof here. 

\begin{thm}
	Let $\CG$ be a locally compact Hausdorff \'{e}tale groupoid. The following are equivalent.
	\begin{enumerate}[label=(\roman*)]
		\item $\CW(\CG)$ is almost unperforated.
		\item For any $a=(A_1, \dots, A_k)$ and $b=(B_1,\dots, B_l)\in \CK(\CG)$, if $\bigcup_{i=1}^kA_i\subset r(\CG\cdot (\bigcup_{j=1}^lB_j))$ and $\sum_{i=1}^k\mu(A_i)<\sum_{j=1}^l\mu(B_j)$ for all $\mu\in \CD_R(\CG)$ with $\sum_{j=1}^l\mu(B_j)=1$, then  one has \[\bigsqcup_{i=1}^k A_i\times \{i\}\prec \bigsqcup_{j=1}^l B_j\times \{j\}.\]
	\end{enumerate}
\end{thm}

We end this section by the following remark on regular groupoid dimension functions.

\begin{Rmk}
Let $\CG$ be a locally compact Hausdorff \'{e}tale groupoid. It is natural to ask the relation between regular groupoid dimension functions and regular Borel $\CG$-invariant premeasures or measures on $\GU$. It is straightforward to see that any Borel regular $\CG$-invariant measure on $\GU$ is a regular groupoid dimension function. However, it is not known in general whether a regular groupoid dimension function $\mu$ extends to a regular $\CG$-invariant measure. Nevertheless, in the case that $\GU$ is  metrizable, the virtually identical proof of  Lemma 3.6 in \cite{M3} shows that any regular groupoid dimension function can be extended uniquely to a $\CG$-invariant Borel regular premeasure on $\GU$. This proof is routine but quite long. Unlike the setting of compact Hausdorff space in \cite{M3},  our underlying space $\GU$ is locally compact Hausdorff, which means that closed sets are not necessarily compact any more. Nevertheless, all closed sets and open sets are still $\sigma$-compact and so are sets in the set \textit{algebra} $\CA_0$ generated by them when $\GU$ is metrizable. Denote by $K_\sigma$ the collection of all $\sigma$-compact sets in $\GU$. Then replace ``closed sets'' by ``compact sets'' and ``$F_\sigma$'' by ``$K_\sigma$'' respectively in the proof of Lemma 3.6 in \cite{M3}, one would obtain this desired extension result. Then $\mu$ could be uniquely extended to a $\CG$-invariant Borel regular measure on $\GU$ by the classical theorem of Carath\'{e}odory when $\mu$ is finite on all compact sets.  On the other hand, the extension to a Borel measure is also unique for any non-trivial regular groupoid dimension function when $\CG$ is minimal and $\GU$ is compact. In this case, a standard rescaling process allows us to use all probability $\CG$-invariant measures in $M(\CG)$ to determine all regular groupoid dimension functions.
\end{Rmk}

\section{Applications of the groupoid semigroup}
In this section, we use the groupoid semigroup to study pure infiniteness and paradoxical comparison of groupoids. Let $\CG$ be a locally compact Hausdorff \'{e}tale groupoid. First let $O$ be a non-empty open set in $\GU$. Observe that $O\prec_{\CG, 2} O$ is equivalent to the condition $2[(O)]\leq [(O)]$ in $\CW(\CG)$. Therefore,  $\CG$ has paradoxical comparison if and only if $2[(O)]\leq [(O)]$ in $\CW(\CG)$ for any non-empty open set $O$ in $\GU$ if and only if $2[a]\leq [a]$ in $\CW(\CG)$ for any $a\in \CK(\CG)$. This implies that $\CG$ has paradoxical comparison if and only if $\CW(\CG)$ is purely infinite.

\begin{thm}
Let $\CG$ be a locally compact Hausdorff \'{e}tale groupoid. Consider the following properties.
\begin{enumerate}[label=(\roman*)]
	\item $\CG$ has groupoid comparison and $M(\CG)=\emptyset$.
	
	\item $\CG$ is purely infinite.
	
	\item $\CG$ has paradoxical comparison.
	
	\item $\CG$ is weakly purely infinite.
\end{enumerate}
Then (i)$\Rightarrow$(ii)$\Leftrightarrow$(iii)$\Rightarrow$(iv). If $\CG$ is minimal then they are equivalent.
\end{thm}
\begin{proof}
(i)$\Rightarrow$(ii) has been established by Lemma 3.10. 

(ii)$\Leftrightarrow$(iii). It suffices to show ``$\Leftarrow$'' because the direction  ``$\Rightarrow$'' is trivial. First note that $\CW(\CG)$ is purely infinite since $\CG$ has paradoxical comparison. Then Proposition 4.2 implies that $\CW(\CG)$ is almost unperforated. Now suppose $U, V$ are non-empty open sets in $\GU$ such that $U\subset r(\CG V)$. Define $a=(U, U), b=(V)\in \CK(\CG)$. Now let $c=(O_1, O_2)\leq_\CG a$ in $\CK(\CG)$. Then recall that each $\overline{O_i}$ is compact and satisfies $O_i\subset \overline{O_i}\subset U$. Since $U\subset r(\CG V)$ one has $\overline{O_i}\subset r(\CG V)$ for each $i=1, 2$. Then for any $u\in \overline{O_i}$ there is a $\gamma\in \CG$ such that $s(\gamma)=u$ and $r(\gamma) \in V$. Since $\CG$ is \'{e}tale, there is a precompact open bisection $W_u$ such that $u\in s(W_u)$ and $r(W_u)\subset V$. Then the compactness of $\overline{O_i}$ implies that there are finitely many such precompact open bisections $W^i_1,\dots, W^i_{n_i}$ such that $\overline{O_i}\subset \bigcup_{j=1}^{n_i} s(W^i_j)$ and $\bigcup_{j=1}^{n_i} r(W^i_j)\subset V$. Then for each $i=1, 2$ one has
\[(O_i)\preccurlyeq (\underbrace{V, \dots , V}\limits_{n_i\  \textrm{many}})\] and thus 
\[c=(O_1, O_2)\preccurlyeq (\underbrace{V, \dots , V}\limits_{n\  \textrm{many}}),\] where $n=n_1+n_2$. Therefore $[c]\leq n[b]$ in $\CK(\CG)$. Then since $\CW(\CG)$ is purely infinite, there is no non-trivial state for $\CW(\CG)$ by Proposition 4.2 and thus Proposition 4.1(iii)$\Rightarrow$(i) implies that there is an $m\in \N^+$ such that $(m+1)[c]\leq m[b]$. Because $\CW(\CG)$ is almost unperforated, one has $[c]\leq [b]$, which means $c\preccurlyeq b$ in $\CK(\CG)$. Since $c\leq_\CG a$ is arbitrary, one has $a\preccurlyeq b$ by Proposition 4.6. On the other hand $a\preccurlyeq b$ means $U\prec_{\CG, 2} V$. This shows that $\CG$ is purely infinite.

(iii)$\Rightarrow$(iv). It suffices to show (ii)$\Rightarrow$(iv). This holds because $U\prec_{\CG, 2} V$ automatically implies $U\prec_{\CG} V$ for any non-empty open sets $U, V$ in $\GU$.

(iv)$\Rightarrow$(i) under the assumption of minimality. If $\CG$ is minimal then for any non-empty open sets $U, V$ in $\GU$ one has $U\subset r(\CG V)=\GU$ trivially and thus if $\CG$ is weakly purely infinite then $U\prec_{\CG} V$ holds for any non-empty open sets $U, V$.  This establishes (i) by the remark before Lemma 3.10.
\end{proof}

If the groupoid is not minimal, it is not true in general that the weakly pure infiniteness implies the pure infiniteness. For example, consider the trivial group $\Gamma=\{e\}$ acting on a compact Hausdorff space $X$.  Then the transformation groupoid of this system is weakly purely infinite but not purely infinite. However, every point $x\in X$ in this system is a global fixed point, i.e., $\Gamma\cdot \{x\}=\{x\}$. In the groupoid case, we call a unit $u\in \GU$ a \textit{global fixed} unit if $s(\gamma)=u$ implies $r(\gamma)=u$ for any $\gamma\in \CG$. We conjecture that our weakly pure infiniteness implies the pure infiniteness for locally compact Hausdorff \'{e}tale groupoids having no global fixed unit in the sense that for any $u\in \GU$ there is a $\gamma\in \CG$ such that $s(\gamma)=u$ and $r(\gamma)\neq u$. The following, as a partial evidence,  shows that if a weakly purely infinite locally compact Hausdorff \'{e}tale groupoid $\CG$ has no global fixed unit then it admits no non-trivial regular groupoid dimension function on $\GU$ and thus its groupoid semigroup $\CW(\CG)$ admits no non-trivial lower semi-continuous state. This makes $\CG$ have the flavor of infiniteness

\begin{prop}
	Let $\CG$ be a locally compact Hausdorff \'{e}tale groupoid. Suppose for any $u\in \GU$ there is a $\gamma\in \CG$ such that $s(\gamma)=u$ and $r(\gamma)\neq u$. If $\CG$ is weakly purely infinite then it admits no non-trivial regular groupoid dimension function.
\end{prop}
\begin{proof}
	Suppose there is a non-trivial regular groupoid dimension function $\mu$. Then there is an open set $U$ in $\GU$ such that $0<\mu(U)<\infty$. Then since $\mu$ is regular, there is a compact set $K\subset U$ such that $0<\mu(K)<\infty$. Then for each $u\in K$ there is a $\gamma\in \CG$ such that $u=s(\gamma)\neq r(\gamma)$.  Then there is an open bisection $O_u$ such that $\gamma\in O_u$ with the property $s(O_u)\subset U$ and $s(O_u)\cap r(O_u)=\emptyset$. Note that $s(O_u)$ is an open neighborhood of $u$ and this implies $K\subset \bigcup_{u\in K}s(O_u)$.  Since $K$ is compact, there are finitely many open bisections $O_1,\dots, O_n$ such that $K\subset \bigcup_{i=1}^ns(O_i)\subset U$ and $s(O_i)\cap r(O_i)=\emptyset$ for each $i\leq n$. Then one has 
	\[0<\mu(K)\leq \sum_{i=1}^n\mu(s(O_i))\leq n\mu(U)<\infty.\] This implies that there is an $i\leq n$ such that $0<\mu(s(O_i))<\infty$. Then because $s(O_i)\sqcup r(O_i)\subset r(\CG s(O_i))$, one has $s(O_i)\sqcup r(O_i)\prec_{\CG}  s(O_i)$ since $\CG$ is weakly purely infinite.  Then since $\mu$ is regular, one has $2\mu(s(O_i))=\mu(s(O_i)\sqcup r(O_i))\leq \mu(s(O_i))$. A contradiction to the fact $0<\mu(s(O_i))<\infty$.
\end{proof}

In the rest of this section, we will focus on ample groupoids. First, if the groupoid $\CG$ in Proposition 5.2 is furthermore assumed to be ample then weakly pure infiniteness of $\CG$ indeed implies that $\CG$ has paradoxical comparison and thus confirms the conjecture above in the ample case.

\begin{prop}
	Let $\CG$ be a locally compact ample Hausdorff \'{e}tale groupoid. Suppose for any $u\in \GU$ there is a $\gamma\in \CG$ such that $s(\gamma)=u$ and $r(\gamma)\neq u$. If $\CG$ is weakly purely infinite then $\CG$ has paradoxical comparison.
\end{prop}
\begin{proof}
	It suffices to show $P\prec_{\CG, 2} P$ for any non-empty compact open set $P$ by Proposition 3.8.  For any $u\in P$ there is a $\gamma\in \CG$ such that $u=s(\gamma)\neq r(\gamma)$. Then there is a compact open bisection $O_u$ such that $\gamma\in O_u$ and $s(O_u)\cap r(O_u)=\emptyset$. Note that $s(O_u)$ is an open neighborhood of $u$ and thus $P\subset \bigcup_{u\in P}s(O_u)$. Then there is a finite subcollection $\{O_1, \dots, O_n\}$ of $\{O_u: u\in P\}$ such that $P\subset \bigcup_{i=1}^ns(O_i)$ since $P$ is compact. Then there is a family $\{V_1, \dots, V_n\}$ of compact open sets such that $V_i\subset s(O_i)$ for each $i\leq n$ and $P=\bigsqcup_{i=1}^n V_i$ (some of $V_i$ could be empty). Define $U_i=(s|_{O_i})^{-1}(V_i)$ for each $i\leq n$. Then $U_1, \dots, U_n$ are compact open bisections such that $\{s(U_1),\dots, s(U_n)\}$ is disjoint and $P=\bigsqcup_{i=1}^ns(U_i)$ and $s(U_i)\cap r(U_i)=\emptyset$ for each $i\leq n$. Now since $\CG$ is weakly purely infinite, for each $i\leq n$ one has $s(U_i)\sqcup r(U_i)\prec_{\CG} s(U_i)$, which implies that $[(s(U_i)\sqcup r(U_i))]\leq [(s(U_i))]$ in $\CW(\CG)$. This implies that for each $i\leq n$ one has
	\[2[s(U_i)]=[(s(U_i), s(U_i))]=[(s(U_i), r(U_i))]=[(s(U_i)\sqcup r(U_i))]\leq [(s(U_i))].\] Then one has 
	\[2[(P)]=2\sum_{i=1}^n[(s(U_i))]=\sum_{i=1}^n2[(s(U_i))]\leq\sum_{i=1}^n[(s(U_i))]=[(P)].\] This shows that $P\prec_{\CG, 2} P$ and thus $\CG$ has paradoxical comparison.
	\end{proof}

The following shows that  one bisection is enough to describe the subequivalence relation ``$\prec_{\CG}$'' if the groupoid $\CG$ is ample.

\begin{prop}
	Let $\CG$ be a locally compact Hausdorff \'{e}tale ample groupoid. Let $U$ be a non-empty compact open set and $V$ a non-empty open set in $\GU$. Then $U\prec_{\CG} V$ if and only if there is a compact open bisections $O$ such that 
	$U=s(O)$ and $r(O)\subset V$.
\end{prop}
\begin{proof}
	It suffices to show the ``only if'' part because the ``if'' part is trivial. Suppose $U\prec_{\CG} V$. Then there are finitely many open bisections $O_1,\dots, O_n$ in $\CG$ such that 
	$U\subset \bigcup_{i=1}^n s(O_i)$ and $\bigsqcup_{i=1}^nr(O_i)\subset V$. Now since $\CG$ is locally compact ample and $U$ is compact, there are disjoint compact open sets $U_1, \dots, U_n$ such that $U_i\subset s(O_i)$ for each $i\leq n$ and $U=\bigsqcup_{i=1}^nU_i$ (some may be empty). Now define $O=\bigsqcup_{i=1}^n (s|_{O_i})^{-1}(U_i)$, which is a compact open bisection such that $U=s(O)$ and $r(O)\subset \bigsqcup_{i=1}^nr(O_i)\subset V$.
\end{proof}

\begin{cor}
	Let $\CG$ be a locally compact Hausdorff \'{e}tale ample groupoid.  Then $\CG$ is purely infinite in the sense of Definition 3.5 if and only if $\CG$ is purely infinite in the sense of Definition 4.9 in \cite{Matui2} by Matui.
\end{cor}
\begin{proof}
	Suppose $\CG$ is purely infinite in the sense of Definition 3.5. Then $\CG$ has paradoxical comparison, which implies that $O\prec_{\CG, 2} O$ for any compact open set $O$ in $\GU$. Since $O$ is compact,  there are disjoint non-empty open sets $O_1, O_2\subset O$ such that $O\prec_{\CG} O_1$ and $O\prec_{\CG} O_2$. Then Proposition 5.4 implies that there are compact open bisections $U, V$ such that $O=s(U)=s(V)$ and $r(U)\subset O_1$ and $r(V)\subset O_2$. Thus $r(U)\cap r(V)=\emptyset$ and $r(U)\sqcup r(V)\subset O$. This is exacly the definition of pure infiniteness in the sense of Matui.
	
	For the reverse direction. Matui's pure infiniteness directly implies that $O\prec_{\CG, 2} O$ for any compact open set $O$ in $\GU$. Then Proposition 3.8 implies that $\CG$ has paradoxical comparison and thus is purely infinite in our sense by Theorem 5.1.
\end{proof}

We then discuss the relation among our pure infiniteness, $n$-filling  and local contraction property described in Example 3.15, 3.16, respectively.

\begin{cor}
	Let $\CG$ be a locally compact Hausdorff \'{e}tale minimal ample groupoid on a compact space. Then $\CG$ has groupoid comparison and $M(\CG)=\emptyset$ if and only if $\CG$ is $1$-filling.
\end{cor}
\begin{proof}
Let $U$ be a non-empty open set in $\GU$. Now suppose $\CG$ has groupoid comparison and $M(\CG)=\emptyset$. Then one has $\GU\prec_{\CG} U$. Then Proposition 5.4 implies that there is a compact open bisection $O$ such that $s(O)=\GU$ and $r(O)\subset U$. This implies $r(O^{-1}U)=\GU$. Therefore, $\CG$ is $1$-filling. The reverse direction has been established in Example 3.15.
\end{proof}

\begin{cor}
		Let $\CG$ be a locally compact Hausdorff \'{e}tale minimal ample groupoid. Suppose $\CG$ has groupoid comparison and $M(\CG)=\emptyset$. Then $\CG$ is a locally contracting groupoid.
\end{cor}
\begin{proof}
	Let $U$ be a non-empty open set in $\GU$ and $V$ a compact open non-empty subset of $U$. Choose another non-empty open set $V_1\subsetneq V$. Now since $\CG$ has groupoid comparison and $M(\CG)=\emptyset$, Proposition 5.4 shows that there is a compact open bisection $O$ such that $V=s(O)$ and $r(O)\subset V_1\subsetneq V$. This shows that $\CG$ is locally contracting.
\end{proof}

To summarize, we have the following result.

\begin{thm}
	Let $\CG$ be a locally compact Hausdorff \'{e}tale minimal ample groupoid on a compact space. Consider the following properties
	\begin{enumerate}[label=(\roman*)]
		\item $\CG$ has groupoid comparison and $M(\CG)=\emptyset$.
		\item $\CG$ is purely infinite (no matter in which sense).
		\item $\CG$ is $n$-filling for some $n\in \N^+$.
		\item $\CG$ is $1$-filling.
		\item $\CG$ is locally contracting.
	\end{enumerate}
Then (i)-(iv) are equivalent and imply (v).
\end{thm}

This partially answers a question in \cite{Ra-Sims} asking how the $n$-filling relates to the locally contracting property in the case that groupoid is ample and minimal. 

\begin{Rmk}
	Let $\CG_{\Gamma\curvearrowright X}$ be a transformation groupoid generated by $\Gamma\curvearrowright X$. We remark that $n$-filling of $\CG_{\Gamma\curvearrowright X}$ is not equivalent to the $n$-filling of $\Gamma\curvearrowright X$ in the sense of \cite{J-R} in which a lot of $n$-filling Cantor dynamical systems were presented such that $n$ cannot be one. In fact, a $1$-filling dynamical system is trivial.  However, the transformation groupoid of any of these systems is $1$-filling by Theorem 5.8. This mainly because a bisection in a transformation groupoid could involve arbitrarily many group elements.
\end{Rmk}

Now we turn to the type semigroup of an ample groupoid. The study of the type semigroup dates back to Tarski, who used this algebraic tool to study paradoxical
decompositions. In the context of topological dynamics on totally disconnected spaces, so far many authors have studied this topic, for example,  \cite{D}, \cite{M2}, \cite{R-S} and \cite{Wagon}.  B\"{o}nicke-Li in \cite{B-L}  and Rainone-Sims in \cite{Ra-Sims}, independently, generalized this semigroup to the setting of locally compact Hausdorff \'{e}tale ample groupoids. We briefly recall the definition here. Let $\CG$ be a locally compact Hausdorff \'{e}tale ample groupoid. Denote by $\CC(\CG)=\{a=(A_1,\dots, A_n): A_i\  \text{compact open for all}\ i\leq n, n\in \N\}$. Let $a=(A_1,\dots, A_n),  b=(B_1, \dots, B_m)\in \CC(\CG)$. Define addition on $\CC(\CG)$ by $a+b=(A_1,\dots, A_n, B_1, \dots, B_m)$. Now we define an equivalence relation on $\CC(\CG)$, which is equivalent to its original definition in Section 5 in \cite{B-L}.

\begin{defn}
	Let $\CG$ be a locally compact Hausdorff \'{e}tale ample groupoid. Define a relation $\sim_\CG$ on $\CC(\CG)$ by announcing $a=(A_1,\dots, A_n)\sim_\CG b=(B_1, \dots, B_m)$ if for each $i\leq n$ there are compact open bisections $W^{(i)}_1, \dots, W^{(i)}_{J_i}$ in $\CG$ and integers $k^{(i)}_1, \dots, k^{(i)}_{J_i}\in \{1, \dots, m\}$ such that $A_i=\bigsqcup_{j=1}^{J_i}s(W^{(i)}_j)$ for each $i\leq n$ and 
	\[\bigsqcup_{l=1}^m B_l\times \{l\}=\bigsqcup_{i=1}^n\bigsqcup_{j=1}^{J_i}r(W_j^{(i)})\times \{k^{(i)}_j\}.\]
\end{defn}

It is not hard to verify that the relation $\sim_\CG$ above is an equivalence relation (for example, see \cite{B-L}). Then the type semigroup $\CV(\CG)$ is defined to be $\CV(\CG)=\CC(\CG)/\sim_\CG$ with the addition $[a]+[b]=[a+b]$. In addition equip $\CV(\CG)$ with the algebraic preorder, i.e. $x\leq y$ in $\CV(\CG)$ if $y=x+z$ for some $z\in \CV(\CG)$. Let $\CG$ be a locally compact Hausdorff \'{e}tale ample groupoid. Then there is a natural map $\kappa$ from $\CV(\CG)$ to $\CW(\CG)$ defined by $\kappa: [a]_{\CV(\CG)}\mapsto [a]_{\CW(\CG)}$ for any $a\in \CC(\CG)\subset \CK(\CG)$.  By definition, $\kappa$ preserves the addition operation and neutral elements of the monoids.  We show below $\kappa$ perserves orders and thus $\kappa$ is a preordered commutative monoid morphism from  $\CV(\CG)$ to $\CW(\CG)$. This implies that  $\kappa$ is a groupoid analogue of the natural map from the Murray-von Neumann semigroup to the Cuntz semigroup in the $C^*$-setting (see \cite{A-P-T} for example).

\begin{prop}
	Let $\CG$ be a locally compact Hausdorff \'{e}tale ample groupoid. Let $a, b\in \CC(\CG)$. Then $[a]_{\CV(\CG)}\leq [b]_{\CV(\CG)}$ if and only if $[a]_{\CW(\CG)}\leq [b]_{\CW(\CG)}$.
\end{prop} 
\begin{proof}
	Let $a=(A_1, \dots, A_n)$ and $b=(B_1, \dots, B_m)$ be elements in $\CC(\CG)$. Suppose $[a]_{\CV(\CG)}\leq [b]_{\CV(\CG)}$. Then there is a $c=(C_1, \dots, C_l)\in \CC(\CG)$ such that $a+c\sim_\CG b$, where $a+c=(A_1, \dots, A_n, C_1,\dots, C_l)$. Then this implies that if for each $i\leq n$ and every compact set $F_i\subset A_i$ there are compact open bisections $W^{(i)}_1, \dots, W^{(i)}_{J_i}$ in $\CG$ and $k^{(i)}_1, \dots, k^{(i)}_{J_i}\in \{1, \dots, m\}$ such that $F_i\subset A_i=\bigsqcup_{j=1}^{J_i}s(W^{(i)}_j)$ for each $i\leq n$ and 
	\[\bigsqcup_{i=1}^n\bigsqcup_{j=1}^{J_i}r(W_j^{(i)})\times \{k^{(i)}_j\}\subset \bigsqcup_{l=1}^m B_l\times \{l\}.\]
	This shows that $a\preccurlyeq b$ in $\CK(\CG)$ and thus one has  $[a]_{\CW(\CG)}\leq [b]_{\CW(\CG)}$. 
	
	For the reverse direction, suppose $a\preccurlyeq b$ in $\CK(\CG)$. Then for every $i\in\{1,\dots, n\}$, since $A_i$ is compact open, there are a collection of open bisections $U^{(i)}_1,\dots, U^{(i)}_{J_i}$ in $\CG$ and $k_1^{(i)},\dots, k_{J_i}^{(i)}\in \{1,\dots, m\}$ such that
	$A_i\subset \bigcup_{j=1}^{J_i}s(U^{(i)}_j)$ and
	\[\bigsqcup_{i=1}^n\bigsqcup_{j=1}^{J_i}r(U^{(i)}_j)\times \{k^{(i)}_j\}\subset \bigsqcup_{l=1}^m B_l\times \{l\}.\]
	Then since $\CG$ is ample, for each $i\leq n$, there is a disjoint collections $\{W_j^{(i)}\subset U_j^{(i)}: 1\leq j\leq J_i\}$ of compact open bisections such that $A_i=\bigsqcup_{j=1}^{J_i}s(W_j^{(i)})$ and 
	\[\bigsqcup_{i=1}^n\bigsqcup_{j=1}^{J_i}r(W^{(i)}_j)\times \{k^{(i)}_j\}\subset \bigsqcup_{l=1}^m B_l\times \{l\}.\]
	Now for each $l\leq m$ define 
	\[C_l=B_l\setminus \bigsqcup \{r(W^{(i)}_j): 1\leq i\leq n, 1\leq j\leq J_i, k^{(i)}_j=l\},\]
	which is also a compact open set (may be empty). Then define $c=(C_1, \dots, C_m)$, which is an element in $\CC(\CG)$. It is not hard to see $a+c\sim_\CG b$ and thus one has $[a]_{\CV(\CG)}\leq [b]_{\CV(\CG)}$.
	\end{proof}

Let $\CG$ be a locally compact Hausdorff \'{e}tale ample groupoid. By the same construction, any groupoid dimension function $\mu$ on $\GU$ induces a state $T_\mu$ on $\CV(\CG)$ by $T_\mu([a])=\sum_{i=1}^n\mu(A_i)$ for $a=(A_1,\dots, A_n)\in \CC(\CG)$. We say a groupoid dimension function $\mu$ on the unit space $\GU$ of a locally compact Hausdorff \'{e}tale groupoid $\CG$ is \textit{faithful} if $\mu(A)>0$ whenever $A$ is a non-empty open set in $\GU$. If $\CG$  is minimal and $\GU$ is compact then any non-trivial regular Borel $\CG$-invariant measure on $\GU$ is a regular faithful groupoid dimension function. Observe that if $\mu$ is faithful then the induced state $T_\mu$ above is also faithful in the same sense that $T_\mu([a])>0$ whenever $[a]\neq 0_{\CV(\CG)}$. Then, similar to the $C^*$-setting, we have the following result.

\begin{prop}
Let $\CG$ be a locally compact Hausdorff \'{e}tale ample groupoid. Suppose there is a faithful groupoid dimension function $\mu$ on $\GU$. Then the morphism $\kappa: \CV(\CG)\to \CW(\CG)$ is an order preserving embedding.
\end{prop}
\begin{proof}
	Let $a ,b\in \CC(\CG)$ such that $[a]_{\CW(\CG)}= [b]_{\CW(\CG)}$. Then Proposition 5.11 implies that there are $c_1, c_2\in \CC(\CG)$ such that  $a+c_1\sim_\CG b$ and $b+c_2\sim_\CG a$. This implies that $a+b+c_1+c_2\sim_\CG a+b$. Now since $\mu$ is faithful, the induced state $T_\mu$ is also faithful. Now one has 
	\[T_\mu([a]_{\CV(\CG)}+[b]_{\CV(\CG)}+[c_1]_{\CV(\CG)}+[c_2]_{\CV(\CG)})=T_\mu([a]_{\CV(\CG)}+[b]_{\CV(\CG)})\] and thus
	$T_\mu([c_1]_{\CV(\CG)})=T_\mu([c_2]_{\CV(\CG)})=0$. This shows $[c_1]_{\CV(\CG)}=[c_2]_{\CV(\CG)}=0_{\CV(\CG)}$ and thus $[a]_{\CV(\CG)}=[b]_{\CV(\CG)}$ because $[a]_{\CV(\CG)}+[c_1]_{\CV(\CG)}=[b]_{\CV(\CG)}$. This entails that $\kappa$ is injective and thus an order preserving embedding.
 \end{proof}

Then we have the following result on almost unperforation of $\CV(\CG)$ and $\CW(\CG)$ for ample groupoids.

\begin{prop}
	Let $\CG$ be a locally compact Hausdorff \'{e}tale ample groupoid. If $\CW(\CG)$ is almost unperforated then so is $\CV(\CG)$.
\end{prop}
\begin{proof}
Let $[a]_{\CV(\CG)}, [b]_{\CV(\CG)}\in \CV(\CG)$ be such that $(n+1)[a]_{\CV(\CG)}\leq n[b]_{\CV(\CG)}$ for some $n\in \N^+$. Then Proposition 5.11 implies $(n+1)[a]_{\CW(\CG)}\leq n[b]_{\CW(\CG)}$ and thus one has $[a]_{\CW(\CG)}\leq [b]_{\CW(\CG)}$ because $\CW(\CG)$ is almost unperforated. Then Proposition 5.11 again implies $[a]_{\CV(\CG)}\leq [b]_{\CV(\CG)}$. This shows that $\CV(\CG)$ is almost unperforated.
\end{proof}

\begin{prop}
Let $\CG$ be a locally compact Hausdorff \'{e}tale ample groupoid. Then $\CW(\CG)$ is purely infinite if and only if $\CV(\CG)$ is purely infinite.
\end{prop}
\begin{proof}
	Proposition 5.11 implies that if $\CW(\CG)$ is purely infinite then so is $\CV(\CG)$. For the reverse direction, suppose $\CV(\CG)$ is purely infinite. Then for any compact open set $O$ one has $2[(O)]_{\CV(\CG)}\leq [(O)]_{\CV(\CG)}$. This implies that $2[(O)]_{\CW(\CG)}\leq [(O)]_{\CW(\CG)}$ by Proposition 5.11 again. Therefore one has $O\prec_{\CG, 2} O$. Then Proposition 3.8 shows that $\CG$ has paradoxical comparison and therefore $\CW(\CG)$ is purely infinite by the observation in the paragraph before Theorem 5.1. 
\end{proof}

Now we have the following result for ample groupoids.

\begin{thm}
	Let $\CG$ be a locally compact Hausdorff \'{e}tale ample groupoid.  Consider the following conditions.
	\begin{enumerate}[label=(\roman*)]
		\item $\CG$ is purely infinite.
		
		\item $\CG$ has paradoxical comparison.
		
		\item Every clopen set  in $\GU$ is $(2, 1)$-paradoxical in the sense of \cite{B-L}.
		
		\item $\CW(\CG)$ is purely infinite.
		
		\item $\CV(\CG)$ is purely infinite.
		
		\item $\CW(\CG)$ is almost unperforated and there is no non-trivial state on $\CW(\CG)$.
		
	\item $\CV(\CG)$ is almost unperforated and there is no non-trivial state on $\CV(\CG)$.
	
	\item $\CG$ is weakly purely infinite.
	\end{enumerate}
Then (i)-(vii) are equivalent in general. If $\CG$ is furthermore assumed to have no global fixed unit, then all conditions above are equivalent to (viii).
\end{thm}
\begin{proof}
	This is a direct application of Remark 3.9,  Proposition 4.2, Theorem 5.1, Proposition 5.3 and Proposition 5.14.
\end{proof}

\begin{Rmk}
	Most recently, Ara, B\"{o}nicke, Bosa and Li posted a paper \cite{A-B-B-L} in which the equivalence of (v) and (viii) above has also been established by an algebraic argument under the assumption that the ample groupoid $\CG$ above is second countable and there is no non-trivial $\CG$-invariant Borel measure on $\GU$ (see Proposition 2.11 in \cite{A-B-B-L}). Note that when there is no non-trivial $\CG$-invariant Borel measure on $\GU$, the dynamical comparison they used there is exactly same to our weakly pure infiniteness. See Definition 2.1 in \cite{A-B-B-L} and also Definition 1.3 in \cite{M3}. Observe that any global fixed unit $u$ induces a probability Borel $\CG$-invariant measure $\mu=\delta_u$, i.e., the Dirac measure at $u$. Therefore, their assumption that $\CG$ has no non-trivial $\CG$-invariant Borel measures on $\GU$ implies our hypothesis in Theorem 5.15 that $\CG$ has no global fixed unit.  In addition, Theorem 5.15 shows that this equivalence actually also holds for locally compact Hausdorff \'{e}tale ample groupoids that are not second countable.
\end{Rmk}

We end this section by a further analysis of Example 3.28. The following deep result was first established by Tarski and recorded as Theorem 5.1 in \cite{B-L}.

\begin{thm}
		Let $S$ be a preordered commutative monoid equipped with the algebraic preorder $\leq$ and let $x\in S$. Then the following are equivalent.
		\begin{enumerate}[label=(\roman*)]
			\item $(n+1)x\nleq nx$ for any $n\in \N$.
			\item There exists a state $f: S\rightarrow [0, \infty]$ with $f(x)=1$.
		\end{enumerate} 
\end{thm}

\begin{Rmk}
In Example 3.28, one can choose an action $\alpha: \Gamma \curvearrowright X$ such that $\alpha$ is minimal purely infinite and $X$ is the Cantor set to begin with. Theorem 5.1 implies that $M_\Gamma(X)=\emptyset$. There are a lot of such actions provided in \cite{J-R}. Then the underlying space of the extension system $\alpha\times \beta: \Gamma \curvearrowright X\times \Gamma^*$ is still zero-dimensional. First there is no finite $\Gamma$-invariant measure on $X\times \Gamma^*$. Suppose the contrary, the push forward measure on $X$ is a finite $\Gamma$-invariant measure on $X$, which is a contradiction to the fact $M_\Gamma(X)=\emptyset$. Nevertheless, there are still  many non-trivial infinite $\Gamma$-invariant measures on $X\times \Gamma^*$, which implies that the extension $\alpha\times \beta$ still has a flavor of finiteness. Indeed, consider the type semigroup $\CV(\alpha\times \beta)$, which is a preordered commutative monoid equipped with the algebraic order mentioned above. Then for any non-empty compact open set $O\times \{h\}$ in $X\times \Gamma$, since the order in the type semigroup $\CV(\alpha\times \beta)$ coincides with the order in the groupoid semigroup $\CW(\alpha\times \beta)$ established in Proposition 5.11, the same argument in Example 3.28 actually shows that $(n+1)[(O\times \{h\})]\nleq n[(O\times\{h\})]$ in $\CV(\alpha\times \beta)$ for any $n\in \N$. Then Theorem 5.17 implies that  there is a non-trivial state $f\in S(\CV, [(O\times \{h\})])$. In addition, Lemma 5.1 in \cite{R-S} shows that $f$ induces a non-trivial $\Gamma$-invariant measure $\mu: X\times \Gamma^*\rightarrow [0, \infty]$ with $\mu(O\times \{h\})=1$.  This fact naturally leads to the following question.
\end{Rmk}

\begin{ques}
 For any purely infinite dynamical system $\alpha: \Gamma \curvearrowright X$, is there an extension $\beta: \Gamma \curvearrowright Y$ of $\alpha$ such that $\beta$ is not purely infinite and has no $\Gamma$-invariant non-trivial measure? 
\end{ques}

\section{$C^*$-algebras arising from the purely infinite groupoids}
In this section, we study the $C^*$-algebras of  locally compact Hausdorff \'{e}tale groupoids. The following shows that our subequivalence relation ``$\preccurlyeq$'' on $\CK(\CG)$ naturally relates to the Cuntz subequivalence relations on functions in $C_0(\GU)$. Denote by $\diag(a_1, \dots, a_n)$ the diagonal matrix whose  entries on diagonal are $a_1, \dots, a_n$. The following is a groupoid version of Proposition 2.3 in \cite{M3}, which can be established by a virtually identical proof. Therefore, we omit its proof.

\begin{prop}
	Let $(f_1,\dots, f_n)$ and $(g_1,\dots, g_m)$ be two sequences of functions in $C_0(\GU)_+$. Write $A_i=\supp(f_i)$ and $B_l=\supp(g_l)$ for each $i\leq n$ and $l\leq m$. Denoted by $a=(A_1, \dots, A_n)$ and $b=(B_1, \dots, B_m)$. If $a\preccurlyeq b$ in $\CK(\CG)$ then $\diag(f_1,\dots, f_n)\precsim\diag(g_1,\dots, g_m)$
	in the $C^\ast$-algebra $C^*_r(\CG)$.
\end{prop}

The following result is established by B\"{o}nicke and Li (see Proposition 4.1 in \cite{B-L}), which is a generalization for the case of dynamical systems proved by R{\o}rdam and Sierakowski in \cite{R-S}.

\begin{prop}
Let $\CG$ be a locally compact Hausdorff \'{e}tale groupoid and $E: C^*_r(\CG)\to C_0(\GU)$ be the canonical faithful conditional expectation. Suppose $C_0(\CG)$ separates  the ideals of $C^*_r(\CG)$. Then $C^*_r(\CG)$ is purely infinite if and only if all non-zero functions in $C_0(\GU)_+$ are properly infinite in $C_r^*(\CG)$ and $E(a)\precsim a$ for all $a\in C_r^*(\CG)_+$.
\end{prop}

First we show paradoxical comparison implies that all non-zero functions in $C_0(\GU)_+$ are properly infinite in $C_r^*(\CG)$.

\begin{prop}
	Let $\CG$ be a locally compact Hausdorff \'{e}tale groupoid. Suppose $\CG$ has paradoxical comparison. Then all non-zero positive functions on $\GU$ are properly infinite in $C_r^*(\CG)$.
\end{prop}
\begin{proof}
	Let $f\in C_0(\GU)_+$ be a non-zero element. Then $O=\supp(f)$ is a non-empty open set. Since $\CG$ has paradoxical comparison, one  $(O, O)\preccurlyeq (O)$ in $\CK(\CG)$. Then Proposition 6.1 implies $f\oplus f\precsim f$, which means that $f$ is properly infinite in $C_r^*(\CG)$.
\end{proof}

We then turn to show $E(a)\precsim a$ for all $a\in C_r^*(\CG)_+$ in some interesting cases to establish the pure infiniteness of $C_r^*(\CG)$ by Proposition 6.2. We begin with the discussion on $\CG$-invariant closed sets in $\GU$. We denote by $K_\CG(\GU)$ the collection of all $\CG$-invariant closed sets in $\GU$.

\begin{prop}
Let $\CG$ be a locally compact Hausdorff \'{e}tale groupoid.	Then the closed set $\overline{r(\CG u)}$ is $\CG$-invariant for any $u\in \GU$.
\end{prop}
\begin{proof}
	Write $Y=\overline{r(\CG u)}$ and fix a $v\in Y$ and an $\eta\in \CG$ with $s(\eta)=v$. For any open set $O$ with $r(\eta)\in O$ one can choose an open bisection $M$ such that $\eta\in M$ and $r(\eta)\in r(M)\subset O$ because $\CG$ is \'{e}tale. Observe $s(M)\cap r(\CG u)\neq\emptyset$ and thus choose a $\gamma\in \CG$ such that $s(\gamma)=u$ and $r(\gamma)\in s(M)$. In addition, choose an open bisection $W$ such that $\gamma\in W$ and $r(W)\subset s(M)$. Now define $T=M\cdot W$, which is an open bisection such that $u\in s(W)=s(T)$ and $r(T)\subset r(M)\subset O$. This implies that there is a $\zeta\in T$ such that $s(\zeta)=u$ and $r(\zeta)\in O$. Therefore, one has $r(\eta)\in \overline{r(\CG u)}=Y$ and thus $Y$ is $\CG$-invariant.
\end{proof}

\begin{lem}
	Let $V$ be a non-empty open set in $\GU$ and $M$ be a non-empty set in $\GU$. Then $M\subset r(\CG V)$ if and only if for any closed $\CG$-invariant  subset $Y$ of $\GU$, one has $M\cap Y\neq \emptyset$ implying $V\cap Y\neq \emptyset$.
\end{lem}
\begin{proof}
	Suppose $M\subset r(\CG V)$ and $M\cap Y\neq \emptyset$, where $Y$ is a closed $\CG$-invariant subset $Y$.  Let $u\in M\cap Y$. Then $u\in r(\CG V)$ and thus there is a $\gamma\in \CG$ such that $s(\gamma)=u$ and $r(\gamma)\in V$. Since $Y$ is $\CG$-invariant and $u\in Y$, one has $r(\gamma)\in Y$, which entails that $V\cap Y\neq \emptyset$.
	
	For the reverse direction, let $u\in M$ and define $Y=\overline{r(\CG u)}$, which is a closed $\CG$-invariant set in $\GU$ by Proposition 6.4. Then one has $V\cap Y\neq \emptyset$ because $u\in M\cap Y\neq \emptyset$. But this actually implies that $V\cap r(\CG u)\neq \emptyset$ since $V$ is open. Then there is an $\eta\in \CG$ such that $s(\eta)=u$ and $r(\eta)\in V$, which entails that $u\in r(\CG V)$ and thus $M\subset r(\CG V)$ holds.
\end{proof}

\begin{defn}
 An open set $M$ in $\GU$ is said to be \textit{groupoid small} if for any compact sets $C\subset \CG\setminus \GU$  and $K\subset M$ there are a compact set $F$ and an open set $O$ in $\GU$ with $F\subset O\subset M$ such that $K\subset r(\CG\cdot F)$ and $OCO=\emptyset$ 
\end{defn}

Roughly speaking, a groupoid small set $M$ is small in the following sense. Given compact sets $C\subset \CG\setminus \GU$ and $K\subset M$ one can always find a ``topological small'' open set $O\subset M$ in the sense that it is disjoint with its translation $r(CO)$ by $C$. However, there is a ``dynamical large'' compact set $F\subset O$  in the sense that all translations of $F$ in $\GU$ covers $K$.
It is straightforward to see that all open sets $M\subset \GU$ are groupoid small if $\CG$ is  minimal and topological principal. We will show below in Proposition 6.9 a stronger result.
Now, we have the following key lemma.

\begin{lem}
Let $\CG$ be a locally compact Hausdorff \'{e}tale groupoid. Suppose $\CG$ is weakly purely infinite and every open set $U$ in $\GU$ is groupoid small then $E(a)\precsim a$ holds for any $a\in C^*_r(\CG)_+$. 
\end{lem}
\begin{proof}
	Denote by $A= C^*_r(\CG)$ and let $a\in A_+$. Without loss of generality, we can assume $\|a\|=1$. For any $\epsilon>0$, choose a $\delta<2\epsilon/3$ and define an open set \[U=\supp((E(a)-\epsilon)_+)=\{u\in \GU: E(a)(u)>\epsilon\},\] which is a precompact open set in $\GU$ because $E(a)\in C_0(\GU)_+$. Then choose a $c\in C_c(\CG)$ such that $\|c\|\leq 2$ and $\|c-a^{1/2}\|<\delta/8$. This implies that $\|c^**c-a\|<\delta/2$ and $\|c*c^*-a\|<\delta/2$.
	
	Define $b=c^**c$. Proposition 2.6 implies that  $b=\sum_{k=0}^mf_k$, where each $f_k$ is supported on a precompact open bisection $B_k$, i.e., $\overline{\supp(f_k)}\subset B_k$ such that $B_0\subset \GU$ and $B_k\cap \GU=\emptyset$ for all $0<k\leq m$. Define a compact set $C=\bigcup_{k=1}^m\overline{B_k}$. Observe that $C\cap \GU=\emptyset$ because $\GU$ is clopen. Note that $\|E(a)-E(b)\|<\delta/2$. Define an open set 
	\[M=\{u\in \GU: E(b)(u)>\epsilon-\delta/2\}\] and thus one has 
	 \[U\subset \overline{U}\subset \{u\in \GU: E(a)(u)\geq\epsilon\}\subset M.\]
	 Now since $M$ is groupoid small by assumption, there are a compact set $F$ and an open set $O$ in $\GU$ with $F\subset O\subset M$ such that $\overline{U}\subset r(\CG\cdot F)$ and $OCO=\emptyset$. Then choose a function $g\in C_c(\GU)_+$ such that $0\leq g\leq 1$, $\overline{\supp(g)}\subset O$ and $g\equiv 1$ on $F$. Because $\supp(g*f_k*g)\subset OB_kO\subset  OCO=\emptyset$ for any $0<k\leq m$, one has $g*f_k*g=0$ for any $0<k\leq m$.
	Then one has \[g*b*g=g*(\sum_{k=0}^mf_k)*g=gf_0g=gE(b)g.\]
	Now the fact $F\subset M$  implies that for any $u\in F$ one has 
	\[(gE(b)g)(u)=E(b)(u)>\epsilon-\delta/2>\delta\] by our choice of $\delta$. This then entails that 
	\[F\subset \{u\in \GU: (gE(b)g)(u)>\delta\}=\supp((gE(b)g-\delta)_+).\]
	This then shows that $U\subset\overline{U}\subset r(\CG \cdot \supp((gE(b)g-\delta)_+))$ since $\overline{U}\subset r(\CG\cdot F)$. Now since $\CG$ is weakly purely infinite, one has $U\prec_{\CG} \supp((gE(b)g-\delta)_+)$.  Then Proposition 6.1 implies 
	\begin{align*}
	(E(a)-\epsilon)_+\precsim (gE(b)g-\delta)_+=(g*b*g-\delta)_+=(g*c^**c*g-\delta)_+
	\end{align*}
	On the other hand, Lemmas 1.4 and 1.7 in \cite{NCP} imply that 
	\[(g*c^**c*g-\delta)_+\sim (c*g^2*c^*-\delta)_+\precsim (c*c^*-\delta)_+\precsim a.\]
	These show that $(E(a)-\epsilon)_+\precsim a$ and thus $E(a)\precsim a$ because $\epsilon$ can be chosen arbitrarily small.
	\end{proof}

Now we can establish the following result.

\begin{thm}
		Let $\CG$ be a locally compact Hausdorff \'{e}tale groupoid such that $C_0(\GU)$ separates ideals of $C_r^*(\CG)$.   If $\CG$ is purely infinite  and all open sets in $\GU$ are groupoid small then $C^*_r(\CG)$ is purely infinite. 
\end{thm}
\begin{proof}
	This is a direct application of Theorem 5.1,  Proposition 6.2, 6.3 and Lemma 6.7.
\end{proof}

We have shown in Section 3 that all examples of  dynamical systems in \cite{A-D} and \cite{L-S} and all topological principal $n$-filling locally compact Hausdorff \'{e}tale groupoids in \cite{J-R}, \cite{YS} and \cite{Ra-Sims} are minimal and purely infinite. Note that all the examples in \cite{A-D} and \cite{L-S} are topological free as well. In addition, recall for any locally compact Hausdorff \'{e}tale minimal topological principal groupoid $\CG$, all open sets in $\GU$ are groupoid small. Therefore, our Theorem 6.8 above covers these results mentioned above. However,  for a non-minimal essentially principal ample groupoid, it is not known whether all open sets in the unit space are automatically groupoid small. If so, one can also recover results on pure infiniteness in \cite{R-S}, \cite{B-L} and \cite{Ra-Sims} for $C^*$-algebras of ample groupoids because the $(2,1)$-paradoxicality of compact open sets there  is proved by Theorem 5.15 being equivalent to pure infiniteness of the groupoids.

\begin{prop}
	Let $\CG$ be a locally compact Hausdorff essentially principal \'{e}tale groupoid. Suppose there are only finitely many $\CG$-invariant closed sets. Then every open set $O$ in $\GU$ is groupoid small.
\end{prop}
\begin{proof}
	Let $O$ be an open set and $K$ be a compact subset of $O$. Let $C$ be a compact set in $\CG\setminus \GU$. Since $\CG$ is locally compact Hausdorff, one can choose finitely many open bisections $B_1, \dots, B_m\subset \CG\setminus \GU$ and compact set $F_j\subset B_j$ for each $j\leq m$ such that \[C\subset\bigcup_{j=1}^mF_j\subset \bigcup_{j=1}^mB_j.\] 
	Define 
	\[\CI_O=\{Y\in K_\CG(\GU): Y\cap O\neq \emptyset\}.\]
	Note that ``$\subset$'' defines a natural order on $\CI_O$ by announcing $Y_1\leq Y_2$ if $Y_1\subset Y_2$. Since $K_\CG(\GU)$ is finite, one can enumerate all minimal elements $\{Y_1, \dots, Y_n\}$ in $\CI_O$ with respect to the order ``$\subset$''.  Then for any $i\leq n$  observe that $O\cap(Y_i\setminus(\bigcup_{j\neq i}Y_j)\neq \emptyset$. Because if not, $O\cap(Y_i\setminus\bigcup_{j\neq i}Y_j)=\emptyset$ implies that there is a $j\neq i$ such that $O\cap Y_i\cap Y_j\neq\emptyset$. But this implies that $Y_i\cap Y_j\in \CI_O$ and $Y_i\cap Y_j$ is a proper subset of $Y_i$, which is a contradiction to the minimality of $Y_i$ in $\CI_O$. Then define $U_i=O\setminus \bigcup_{j\neq i}Y_j$ and thus $U_i\cap Y_i\neq \emptyset$. Since $\CG$ is essentially principal, for each $i\leq n$ one can choose a $u_i\in U_i\cap Y_i$ with the trivial isotropy. Observe that \[r(Cu_i)\subset r(\bigcup_{j=1}^mF_ju_i)\subset r(\bigcup_{j=1}^mB_ju_i)\subset Y_i\setminus \bigcup_{j\neq i}Y_j\]
	 because $u_i\in Y_i\setminus \bigcup_{j\neq i}Y_j$, which is $\CG$-invariant. In addition, all units $r(\gamma)$ for $\gamma\in \bigcup_{j=1}^mB_j$ with $s(\gamma)=u_i$ are distinct. Indeed, suppose there are $\gamma$ and $\eta\in \bigcup_{j=1}^mB_j$ such that $s(\gamma)=s(\eta)=u_i$ and $r(\gamma)=r(\eta)$. Then, one has $s(\gamma^{-1}\eta)=r(\gamma^{-1}\eta)=u_i$, which implies $\gamma^{-1}\eta=u_i$ because $u_i$ has the trivial isotropy and thus $\gamma=\eta$. Therefore, the units in \[\{r(\gamma): \gamma \in \bigcup_{j=1}^mB_j, s(\gamma)=u_i, i=1,\dots, n\}\cup\{u_i: i=1, \dots, n\}\]
	are distinct. Since $\GU$ is Hausdorff, one can find a family 
	\[\{W_{i, \gamma}: \gamma\in \bigcup_{j=1}^mB_j, s(\gamma)=u_i, i=1,\dots, n\}\cup\{W_i: i=1, \dots, n\}\] of disjoint open sets in $\GU$ such that for any $i=1,\dots, n$ one has $u_i\in W_i$ and $r(\gamma)\in W_{i, \gamma}$ for all $u_i$ and $\gamma \in \bigcup_{j=1}^mB_j$ with $s(\gamma)=u_i$. Now fix an $i\leq n$ and enumerate $\{\gamma\in \bigcup_{j=1}^mB_j: s(\gamma)=u_i\}$ by $\{\gamma^i_1, \dots, \gamma^i_{l_i}\}$. Then for each $\gamma^i_k$, where $k\leq l_i$ one can choose an open bisection $N^i_k$ such that 
	\[\gamma^i_k\in N^i_k\subset \bigcap_{\gamma^i_k\in B_j}B_j \]
	and $r(N^i_k)\subset W_{i, \gamma^i_k}$.  Define $P_i=\bigcap_{k=1}^{l_i}s(N^i_k)$, which is an open neighborhood of $u_i$ in $\GU$.  Since for each $k\leq l_i$, both $N^i_k$ and $\bigcap_{\gamma^i_k\in B_j}B_j$ are bisections,  for any $v\in P_i$, one has 
	\[N^i_kv=\bigcap_{\gamma^i_k\in B_j}B_jv\]
	Write $\{\eta\}=N^i_kv=\bigcap_{\gamma^i_k\in B_j}B_jv$. Note that for each $B_j$ with $\gamma^i_k\in B_j$, the $\eta$ is the only element in $B_j$ such that $s(\eta)=v$ because each $B_j$ is also a bisection. This entails that $B_jv=\{\eta\}$ for any $B_j$ satisfying $\gamma^i_k\in B_j$ and thus in fact one has
	\[N^i_kv=\bigcup_{\gamma^i_k\in B_j}B_jv.\]
	
	Now for each $i\leq n$ define an open neighborhood $Q_i$ of $u_i$ by 
	\[Q_i=\bigcap_{u_i\in s(B_j)}s(B_j)\setminus \bigcup_{u_i\notin s(B_j)}s(F_j)\] and define $V_i=P_i\cap W_i\cap Q_i\cap U_i$, which is still an open neighborhood of $u_i$. Now for each $v\in V_i$ one has 
	\[Cv\subset \bigcup_{j=1}^mF_jv=\bigcup_{u_i\in s(B_j)}F_jv=\bigcup_{u_i\in s(B_j)}B_jv=\bigcup_{k=1}^{l_i}\bigcup_{\gamma^i_k\in B_j}B_jv=\bigcup_{k=1}^{l_i}N^i_kv.\]
	Then this implies that \[r(Cv)\subset r(\bigcup_{k=1}^{l_i}N^i_kv)\subset \bigsqcup_{k=1}^{l_i}W_{i, \gamma^i_k}.\]
	Now for each $i\leq n$ and $u_i\in V_i$, choose a precompact open set $D_i$ such that 
	\[u_i\in D_i\subset \overline{D_i}\subset V_i\]
	 and define $V=\bigsqcup_{i=1}^nV_i$ and $D=\bigsqcup_{i=1}^nD_i$. Note that, for each $\CG$-invariant closed set $Y\in \CI_O$, there is a minimal element $Y_i\in \CI_O$ such that $Y_i\subset Y$. By our construction, $u_i\in D_i\cap Y_i$ and thus $D\cap Y\neq \emptyset$. Then Lemma 6.5 implies that $K\subset O\subset r(\CG\cdot D)\subset r(\CG\cdot \overline{D})$. In addition, since \[r(CV)\subset \bigsqcup_{i=1}^n\bigsqcup\{W_{i, \gamma}: \gamma\in \bigcup_{j=1}^mB_j, s(\gamma)=u_i\},\] which is disjoint from $V$, one has $VCV=\emptyset$. This shows that $O$ is groupoid small.
\end{proof}

Now, we can establish our Corollary 1.1.

\begin{proof}(Corollary 1.1)
	Since $\CG$ is amenable then $C^*_r(\CG)$ is nuclear by a classical result of Tu in \cite{Tu}. In addition, by \cite{B-L}, the amenability of $\CG$ shows that essentially principality of $\CG$ implies that $C_0(\GU)$ separating ideals of $C_r^*(\CG)$. Then Theorem 6.8 and Proposition 6.9 imply that $C^*_r(\CG)$ is purely infinite. In addition, since $C_0(\GU)$ separates ideals of $C_r^*(\CG)$ and there are only finitely many $\CG$-invariant closed sets in $\GU$, there are only finitely many closed ideals in $C_r^*(\CG)$.  This shows that $\operatorname{Prim}(C_r^*(\CG))$ is finite.  Now if $\CG$ is second countable then $C^*_r(\CG)$ is separable. Then Proposition 2.11 in \cite{P-R} implies that $C_r^*(\CG)$ has the ideal property (IP). Now Proposition 2.14 in \cite{P-R} shows that $C_r^*(\CG)$ is strongly purely infinite.
\end{proof}

Now suppose $\CG$ is minimal. We show below $ M(\CG)\neq \emptyset$ implies stably finiteness of $C^*_r(\CG)$.

\begin{lemma}
	Let $\CG$ be a locally compact Hausdorff \'{e}tale minimal topological principal groupoid. Then any $\mu\in M(\CG)$ induces a lower semi-continuous bounded faithful $2$-quasitrace on $A=C_r^*(\CG)$ defined by $\tau_\mu(a)=\int_{\GU}E(a)\diff{\mu}$ for $a\in A_+$.  
\end{lemma}
\begin{proof}
	Let $\mu\in M(\CG)$. It is straightforward to see $\tau_\mu$ is a quasitrace.  Observe that 
	\begin{align*}
	|\tau_\mu(a)-\tau_\mu((a-\epsilon)_+)|\leq \int_{\GU}\|E(a-(a-\epsilon)_+)\|\diff{\mu}
	\leq \int_{\GU}\|(a-(a-\epsilon)_+)\|\diff{\mu}
	\leq \epsilon.
	\end{align*}
	This shows that $\sup_{\epsilon>0}\tau_\mu((a-\epsilon)_+)=\tau_\mu(a)$ and thus $\tau_\mu$ is lower semi-continuous. Now because $\tau_\mu$ is additive for all $a, b\in A_+$, Proposition 2.24 in \cite{B-K} entails that $\tau_\mu$ is actually a $2$-quasitrace.  In addition, since $\|E(a)\|\leq \|a\|$ for any $a\in A_+$, one has $\tau_\mu(a)\leq \|a\|<\infty$ because $\mu$ is a probability measure. This shows $\tau_\mu$ is bounded. Now observe that $C^*_r(\CG)$ is simple because $\CG$ is minimal and topological principal by Corollary 3.14 in \cite{B-L}. Then $\tau_\mu$ is faithful. 
	
\end{proof}

\begin{thm}
Let $\CG$ be a locally compact Hausdorff \'{e}tale minimal topological principal groupoid. Suppose $\CG$ has groupoid comparison.  Then $C^*_r(\CG)$ is either stably finite or strongly purely infinite.
\end{thm}
\begin{proof}
	Lemma 6.10 implies that if $M(\CG)\neq \emptyset$ then $C_r^*(\CG)$ is stably finite by Remark 2.27(viii) in \cite{B-K}. Now, suppose  $M(\CG)= \emptyset$. Since $\CG$ also has groupoid comparison, Theorem 5.1 implies that $\CG$ has paradoxical comparison. Then Corollary 1.2 implies that $C^*_r(\CG)$ is purely infinite and thus strongly purely infinite because $C^*_r(\CG)$ is simple by \cite{Kir-Rord}.
\end{proof}

We remark that our Theorem 6.11 is a generalization of the similar dichotomy obtained in \cite{B-L} and \cite{Ra-Sims}. This is because first we do not assume that the groupoid $\CG$ is ample and the unit space $\GU$ is compact. In addition, for the case that $\CG$ is ample and minimal, Theorem 5.1 and 5.15 imply that the type semigroup $\CV(\CG)$ is almost unperforated and has no non-trivial state if and only if $\CG$ has groupoid comparison and $M(\CG)=\emptyset$. 

In the rest of this section, we construct generic examples of purely infinite locally compact Hausdorff \'{e}tale groupoids and establish our final main result.

\begin{defn}
Let $\CG$ be a locally compact Hausdorff  \'{e}tale groupoid and $X$ be a locally compact Hausdorff space. Define $\CG_+=\CG\times X$ with the product topology. Equipped $\CG_+$ the groupoid operation by announcing
\begin{enumerate}[label=(\roman*)]
	\item $(\gamma, x)$ and $(\eta, y)\in \CG_+$ are composable only when $s(\gamma)=r(\eta)$ in $\CG$ and $x=y$ in $X$. In this case, $(\gamma, x)\cdot (\eta, y)=(\gamma\eta, x)$.
	
	\item $(\gamma, x)^{-1}=(\gamma^{-1}, x)$ for all $(\gamma, x)$.
	\item $\GU_+=\GU\times X$
\end{enumerate} 
\end{defn}

We remark that $\CG_+$ above is also a locally compact Hausdorff \'{e}tale groupoid. It is not hard to see that the collection of precompact open bisections of the form $O\times V$, where $O$ is a precompact open bisection in $\CG$ and $V$ is a precompact open set in $X$, form a base for the topology on $\CG_+$. In addition, it is not hard to verify that if $\CG$ is amenable then $\CG_+$ is also amenable. Note that $C^*_r(\CG_+)\simeq C^*_r(\CG)\otimes C_0(X)$ because $C^*_r(\CG_+)$ can be regarded as the crossed product of the action of $\CG$ on $C_0(\GU\times X)$ in which $C_0(\GU\times X)\simeq C_0(\GU)\otimes C_0(X)$ is a $C_0(\GU)$-algebra and the action on $C_0(X)$ is trivial. See \cite{A-D2} for the detailed construction.

\begin{prop}
Let $\CG$ be a locally compact Hausdorff  \'{e}tale groupoid and $X$ be a locally compact Hausdorff space. Suppose  $\CG$ has groupoid comparison and $M(\CG)=\emptyset$. Then $\CG_+$ has paradoxical comparison. 
\end{prop}
\begin{proof}
	Let $O$ be open in $\GU_+$ and $F\subset O$ be a compact set. Then since $\GU_+$ is locally compact Hausdorff, one can choose open sets $M_i\times N_i\subset \GU_+$ and compact sets $F_i\subset M_i\times N_i$ for $i=1, \dots, n$ such that 
	\[F\subset \bigcup_{i=1}^nF_i\subset\bigcup_{i=1}^n(M_i\times N_i)\subset O.\]
	Now, since $\CG$ has groupoid comparison and $M(\CG)$ is the empty set, $\GU$ is perfect. This allows us to choose a disjoint collection $\{U_{i,j}: 1\leq i\leq n, j=1,2\}$ of non-empty open sets such that $U_{i,j}\subset M_i$ and $M_i\prec_{\CG} U_{ij}$ for all $j=1,2$ and $1\leq i\leq n$. Define $O_j=\bigsqcup_{i=1}^n(U_{ij}\times N_i)\subset O$ for $j=1,2$, which are two disjoint open sets.  Denote by $\pi_{\GU}$ and $\pi_X$ the canonical projection from $\GU\times X$ onto $\GU$ and $X$, respectively. Then $\pi_{\GU}(F_i)\subset M_i$ is a compact set.  Then fix a $j\in\{1, 2\}$. Since $M_i\prec_{\CG} U_{ij}$, there is a family $\{V^i_1, \dots, V^i_{K_i}\}$ such that $\pi_{\GU}(F_i)\subset \bigcup_{k=1}^{K_i}s(V^i_k)$ and $\bigsqcup_{k=1}^{K_i}r(V^i_k)\subset U_{ij}$. Now note that 
	\[F\subset \bigcup_{i=1}^nF_i\subset \bigcup_{i=1}^n(\pi_{\GU}(F_i)\times \pi_X(F_i))\subset\bigcup_{i=1}^n\bigcup_{k=1}^{K_i}(s(V_k^i)\times N_i)\] and
	\[\bigsqcup_{i=1}^n\bigsqcup_{k=1}^{K_i}(r(V_k^i)\times N_i)\subset \bigsqcup_{i=1}^n(U_{ij}\times N_i)=O_j.\]
	This implies that $F\prec_{\CG} O_j$ for $j=1, 2$ because each $V^i_k\times N_i$ is a bisection in $\CG_+$ such that $s(V^i_k\times N_i)=s(V^i_k)\times N_i$ and $r(V^i_k\times N_i)=r(V^i_k)\times N_i$.  Since $F$ is an arbitrary compact subset of $O$ and $O_1, O_2$ are disjoint open subset of $O$, we have $O\prec_{\CG, 2} O$ and thus $\CG_+$ has paradoxical comparison. 
\end{proof}

It can be shown that if $\CG$ is a locally compact Hausdorff amenable minimal topologically principal  \'{e}tale  groupoid and $X$ be a locally compact Hausdorff space then $\CG_+$ is essentially principal amenable and all open sets in $\GU_+$ are groupoid small. Then Proposition 6.2, 6.3 and Lemma 6.7 show that $C_r^*(\CG_+)$ is purely infinite if $\CG$ has groupoid comparison and $M(\CG)$ is the empty set. However, since $C^*_r(\CG_+)\simeq C^*_r(\CG)\otimes C_0(X)$, we have a better result by using Corollary 1.2 and a result due to Kirchberg and Sierakowski in \cite{Kir-S}.

\begin{thm}
	Let $\CG$ be a locally compact Hausdorff minimal topologically principal  \'{e}tale  groupoid and $X$ be a locally compact Hausdorff space. Suppose $\CG$ has groupoid comparison and $M(\CG)=\emptyset$. Then $C^*_r(\CG_+)$ is strongly purely infinite.
\end{thm}
\begin{proof}
Corollary 1.2 implies that $C^*_r(\CG)$ is  simple and purely infinite.  Therefore,  $C^*_r(\CG)$ is strongly purely infinite by Theorem 9.1 in \cite{Kir-Rord}.  Then Theorem 1.3 in \cite{Kir-S} shows that $C^*_r(\CG_+)\simeq C^*_r(\CG)\otimes C_0(X)$ is strongly purely infinite.
\end{proof}

As an application of amplification groupoids, they yield groupoid model for several strongly purely infinite $C^*$-algebras including some projectionless purely infinite $C^*$-algebras. We provide an explicit example below.
Recall that the action $\alpha_0: \Z_2\ast \Z_3 \curvearrowright X=\{0, 1\}^\N$ in Example 3.23 is a dynamical model of $\CO_2$ and $\alpha_0$ is shown to have groupoid comparison and $M_{\Z_2\ast \Z_3}(X)=\emptyset$. Define  an action $\beta: \Z_2\ast \Z_3 \curvearrowright X\times \R$ by $\beta_g(x, y)=((\alpha_{0})_g(x), y)$ for any $g\in \Z_2\ast \Z_3$. Then we have the following result.

\begin{thm}
	The strongly purely infnite $C^*$-algebra $\CO_2\otimes C_0(\R)$ has a dynamical model $\beta: \Z_2\ast \Z_3 \curvearrowright \{0, 1\}^\N\times \R$ such that $\beta$ has paradoxical comparison (and thus purely infinite). In addition, $\CO_2\otimes C_0(\R)$ has no locally contracting groupoid model.
\end{thm}
\begin{proof}
	Write $X=\{0, 1\}^\N$. Note that the action $\beta: \Z_2\ast \Z_3 \curvearrowright X\times \R$ above satisfies that $C_0(X\times \R)\rtimes_r (\Z_2\ast \Z_3)\simeq \CO_2\otimes C_0(\R)$. Therefore $\beta$ is a dynamical model of $\CO_2\otimes C_0(\R)$. In addition, $\CO_2\otimes C_0(\R)$ contains no non-zero projections, so it has no locally contracting groupoid model because all locally contracting groupoid yields an infinite projection for its reduced $C^*$-algebra.
\end{proof}

\section{Acknowledgement}
The author should like to thank Hanfeng Li for the motivating example (Example 3.28), inspiring suggestions and comments.  In addition, he should like to thank David Kerr and Kang Li for invaluable and helpful comments. He also would like to thank Tsz Fun Hung for helpful discussions. Finally, he thank the referees, whose comments and suggestions helped to significantly improve the paper.


\begin{thebibliography}{10}
		\bibitem{A-D}C. Anantharaman-Delaroche. Purely infinite $C^*$-algebras arising from dynamical systems. \textit{Bull. Soc. Math. France} \textbf{125} (1997), 199-225.
		
		\bibitem{A-D2}C. Anantharaman-Delaroche. Some remarks about the weak containment property for groupoids and semigroups. arXiv: 1604.01724.
		
		\bibitem{A-B-B-L}P. Ara, C. B\"{o}nicke, J. Bosa and K. Li. The type semigroup, comparison and almost finiteness for ample groupoids. arXiv:2001.00376.
		
		\bibitem{A-P-T}P. Ara, F. Perera and A. S. Toms. \textit{K-theory for operator algebras. Classification of $C^\ast$-algebras}, page 1-71 in: \textit{Aspects of Operator Algebras and Applications}. P. Ara, F. Lled\'{o}, and F. Perera (eds.). Contemporary Mathematics vol.534, Amer. Math. Soc., Providence RI, 2011.
		
		
		
		
		
		
		\bibitem{B-K}E. Blanchard and E. Kirchberg. Non simple purely infinite $C^*$-algebras: The Hausdorff case. \textit{J. Func. Anal.} \textbf{207} (2004), 461-513.
		
		\bibitem{B-L}C. B\"{o}nicke and K. Li. Ideal structure and pure infiniteness of ample groupoid $C^\ast$-algebras. \textit{Ergod. Th. and Dynam. Sys.} \textbf{40} (1) (2020), 34-63.
		
		
		
		\bibitem{Cu}J. Cuntz. Dimension functions on simple $C^*$-algebras. \textit{Math. Ann.} \textbf{233}  (1978), 145-153.
		
		\bibitem{Cuntz}J. Cuntz. K-theory for certain $C^*$-algebras. \textit{Ann. of Math.} (2) \textbf{113} (1) (1981), 181-197.
		
		\bibitem{J-R}P. Jolissaint and G. Robertson. Simple purely infinite $C^\ast$-algebras and $n$-filling actions. \textit{J. Funct. Anal.} \textbf{175}(2000), 197-213.
		
		
		\bibitem{D}D. Kerr. Dimension, comparison, and almost finiteness. To appear in \textit{J. Eur. Math. Soc.}
		
		\bibitem{K-Rord}E. Kirchberg and M. R{\o}rdam. Non-simple purely infinite $C^\ast$-algebras. \textit{Amer. J. Math.} \textbf{122}(3) (2000), 637-666.
		
		\bibitem{Kir-Rord}E. Kirchberg and M. R{\o}rdam. Infinite non-simple $C^\ast$-algebras: Absorbing the Cuntz algebra $\mathcal{O}_\infty$. \textit{Adv. Math.} \textbf{167} (2002), 195-264.
		
		
		
		\bibitem{Kir-S}E. Kirchberg and A. Sierakowski. Filling families and strong pure infiniteness. arXiv:1503.08519.
		
		\bibitem{L-S}M. Laca and J. Spielberg. Purely infinite $C^\ast$-algebras from boundary actions of discrete groups. \textit{J. Reine. Angew. Math.} \textbf{480}(1996), 125-139.
		
		\bibitem{M2}X. Ma. Comparison and pure infiniteness of crossed products. \textit{Trans. Amer. Math. Soc.} \textbf{372} (2019), no. 10,  7497-7520
		
		\bibitem{M3}X. Ma. A generalized type semigroup and dynamical comparison. To appear in  \textit{Ergod. Th. and Dynam. Sys.} doi:10.1017/etds.2020.28
		
		\bibitem{Matui2}H. Matui. Topological full groups of one-sided shifts of finite type. \textit{J. Reine. Angew. Math.} \textbf{705} (2015), 35-84.
		
		\bibitem{O-P-R}E. Ortega, F. Perera, and M. R{\o}rdam. The corona factorization property, stability, and the Cuntz semigroup of a $C^\ast$-algebra. \textit{Int. Math. Res. Not.} \textbf{2012}(2012),  34-66.
		
		\bibitem{P-R}C. Pasnicu and M. R{\o}rdam. Purely infinite $C^*$-algebras of real rank zero. \textit{J. Reine. Angew. Math.} \textbf{613} (2007), 51-73.
		
		\bibitem{Phillips}N. C. Phillips. A classification theorem for nuclear purely infinite simple $C^*$-algebras. \textit{Doc. Math.} \textbf{5} (2000), 49-114.
		
		\bibitem{NCP}N. C. Philips. Large subalgebras. arXiv: 1408:5546.
		
		\bibitem{Ra-Sims}T. Rainone and A. Sims. A dichotomy for groupoid $C^*$-algebras. \textit{Ergod. Theory Dyn. Sys.} \textbf{40} (2020), 521-563.
		
		\bibitem{Renault}J. Renault. \textit{A Groupoid approach to $C^*$-algebras} (Lecture Notes in Mathematics, 793). Springer, Berlin, 1980.
		
		\bibitem{Rordam}M. R{\o}rdam. \textit{Classification of nuclear, simple $C^\ast$-algebras}. Classification of nuclear $C^\ast$-algebras. Entropy in operator algebras. Encyclopaedia. Math Sci., vol. 126, Springer, Berlin, 2002, pp. 1-145. 
		
		\bibitem{R-S}M. R{\o}rdam and A. Sierakowski. Purely infinite $C^\ast$-algebras arising from crossed products. \textit{Ergod. Th. and Dynam. Sys.} \textbf{32} (2012), 273-293.
		
		
		
		\bibitem{S}A. Sierakowski. The ideal structure of reduced croseed products. M\"{u}nster J. Math. \textbf{3}(2010), 237-262.
		
		\bibitem{Sims}A. Sims. Hausdorff \'{e}tale groupoids and their $C^*$-algebras. arXiv:1710.10897.
		
		\bibitem{Sp}J. Spielberg. Graph-based models for Kirchberg algebras. \textit{J. Operator Th.} \textbf{57} (2007), 347-374.
		
		\bibitem{YS1}Y. Suzuki. Amenable minimal Cantor systems of free groups arising from diagonal actions. \textit{J. Reine. Angew. Math.} \textbf{722} (2017), 183-214.
		
		\bibitem{YS}Y. Suzuki. Construction of minimal skew products of amenable minimal dynamical systems. \textit{Groups. Geom. Dyn.} \textbf{11} (2017), 75-94.
		
		\bibitem{Tu}J.-L. Tu. La conjecture de Baum-Connes pour les feuilletages moyennables, \textit{K-theory} \textbf{17}(1999), 215-264.
		
		\bibitem{Wagon}S. Wagon. \textit{The Banach-Tarski Paradox}. Cambridge University Press, Cambridge, 1993.
	\end{thebibliography}
\end{document}